\documentclass{article}

\usepackage[T2A]{fontenc}
\usepackage[utf8]{inputenc}
\usepackage[tbtags]{amsmath}
\usepackage{mathtools}
\usepackage{amsfonts,amssymb}
\usepackage{amsthm}

\usepackage{comment}
\usepackage{hyperref}


\numberwithin{equation}{section}

\newtheorem{theorem}{Theorem}
\newtheorem{proposition}{Proposition}[section]
\newtheorem{lemma}{Lemma}
\newtheorem{assmp}{Assumption}
\newtheorem{corol}{Corollary}[section]

\newtheorem{defn}{Definition}
\newtheorem{example}{Example}

\newcommand{\Si}[1][\theta]{\sin_{\Omega}{#1}}
\newcommand{\So}[1][\theta^\circ]{\sin_{\Omega^\circ}{#1}}
\newcommand{\Ci}[1][\theta]{\cos_{\Omega}{#1}}
\newcommand{\Co}[1][\theta^\circ]{\cos_{\Omega^\circ}{#1}}

\newcommand{\sinhi}[1][\theta]{\mathrm{sinh}_{\Omega}{#1}}

\newcommand{\coshi}[1][\theta]{\mathrm{cosh}_{\Omega}{#1}}

\newcommand{\R}{\mathbb{R}}
\newcommand{\UU}{{U}}

\DeclareMathOperator{\dom}{dom}
\DeclareMathOperator*{\argmax}{arg\,max}

\providecommand{\keywords}[1]
{
  \small	
  \textbf{{Keywords: }} #1
}
\title{Explicit formulas for extremals in sub-Lorentzian and Finsler problems on 2- and 3-dimensional Lie groups}
\author{E.\,A.~Ladeishchikov, \and L.\,V.~Lokutsievskiy, \and N.\,V.~Prilepin}

\date{May 25, 2024}
\begin{document}

\maketitle
\begin{abstract}
In this paper, we consider the problem of finding geodesics in a series of left-invariant problems endowed with sub-Lorentzian and Finsler structures. Explicit formulas for extremals are obtained in terms of convex trigonometric functions. In the sub-Lorentzian setting, the new trigonometric functions $\cosh_\Omega$ and $\sinh_\Omega$, developed here, prove especially useful; they generalize the classical $\cosh$ and $\sinh$ to the case of an unbounded convex set $\Omega\subset\R^2$.
\end{abstract}


\markright{extremals in sub-Lorentzian and Finsler problems}

\footnotetext[0]{The research of L.V. Lokutsievskiy was carried out at the MCMU of the Steklov Mathematical Institute of the RAS with financial support from the Ministry of Education and Science of Russia (agreement No. 075-15-2025-303). The research of E.A. Ladeishchikov and L.V. Lokutsievskiy was supported by a grant from the Foundation for the Advancement of Theoretical Physics and Mathematics <<BASIS>>.}

\section{Introduction}

Lorentzian geometry is a cornerstone of the mathematical foundation of the theory of relativity. Unlike classical Riemannian geometry, in Lorentzian geometry the ability to join two points by a timelike curve guarantees that its length can be made arbitrarily small. The reason lies in the physical meaning of Lorentzian length: the length of a timelike curve equals the proper time of a traveller moving along it. One can easily devise a curve joining two given timelike points such that the motion along it occurs at a speed arbitrarily close to the speed of light. On the other hand, the proper time between two timelike points in classical Minkowski space~$\R^{3,1}$ is clearly bounded above by the difference of their time coordinates. Thus, a natural description of geodesics on Lorentzian manifolds is as (locally) longest curves, and the geodesic problem can be regarded as an optimal-control problem of finding such longest curves. Methods for explicitly finding geodesics in Riemannian, sub-Riemannian, Finsler, sub-Finsler, Lorentzian and sub-Lorentzian manifolds are very similar: they all begin with explicit integration of the Hamiltonian system given by Pontryagin’s maximum principle (or the Euler–Lagrange equations) and an analysis of the abnormal case. The goal of this work is to develop explicit methods for solving such systems in sub-Lorentzian and Finsler problems on classical spaces.

Active research in sub-Lorentzian geometry began with the pioneering works of Berestovskii and Gichev \cite{Berestovskii}, which studied metrized orders on topological groups, introduced the notion of an anti-metric in general form, and investigated its properties, as well as the papers of Grochowski \cite{Grochowski1,Grochowski2}, which addressed the existence of longest curves, local properties of the sub-Lorentzian metric, reachability sets, and more. We should also mention the works \cite{MarkinaVasiliev,Vasiliev,Markina} by Vasiliev, Markina and co-authors, where sub-Lorentzian structures on anti-de Sitter spaces, on certain Heisenberg-type groups and on the group $SU(1,1)$ are examined. In questions concerning sub-Lorentzian longest curves and the construction of an optimal synthesis, methods of geometric control theory—recently applied intensively to such problems (see, e.g., \cite{Sachkov1,Sachkov2})—prove highly effective.

\medskip

In the present work we obtain explicit formulas for geodesics in left-invariant sub-Lorentzian problems with an arbitrary anti-norm on all three-dimensional unimodular Lie groups $SU(2)$, $SL(2)$, $SE(2)$, $SH(2)$, $\mathbb{H}3$, and on the Lobachevsky plane (corresponding to the group $\mathrm{Aff}+(\R)$). All these formulas are written in terms of the new functions $\cosh_\Omega$ and $\sinh_\Omega$, which conveniently generalize $\cosh$ and $\sinh$. The classical functions $\cosh$ and $\sinh$ serve well for describing geodesics in the case of the standard two-dimensional Lorentz norm $x^2 - y^2$ (determined by a hyperbola in the plane); however, for an arbitrary concave anti-norm on a cone they are unsuitable. Section \ref{sec:convex_trigonometry} introduces the functions $\cosh_\Omega$ and $\sinh_\Omega$, a natural extension of $\cosh$ and $\sinh$ to an arbitrary two-dimensional anti-norm, inheriting many of their key properties. In particular, these functions are extremely useful for expressing geodesics in sub-Lorentzian problems with an arbitrary planar anti-norm.

In addition, we derive explicit formulas for geodesics in the left-invariant Finsler problem on the Heisenberg group $\mathbb{H}_3$ with a Finsler norm that admits generalized spherical coordinates (see Section \ref{sec:Heisenberg_Finsler} for details). For instance, we obtain formulas for geodesics corresponding to a left-invariant $\ell_p$ norm, $1\le p\le\infty$, on $\mathbb{H}_3$. Note that not every three-dimensional Finsler norm allows such coordinates, but the $\ell_p$ norms do.

\medskip

The paper is structured as follows. Section~\ref{sec:general_statements} presents general formulations of sub-Lorentzian and sub-Finsler problems on Lie groups. Section~\ref{sec:concrete_statements} formulates these problems in coordinates for the three settings studied in this work: (1) the Finsler problem on the three-dimensional Heisenberg group; (2) the Lorentz problem on the Lobachevsky plane; (3) the sub-Lorentzian problem on all three-dimensional unimodular Lie groups. Section \ref{sec:optimal_control_problem} casts the problems of finding shortest and longest curves as time-optimal and time-minimal problems,\footnote{The transition to the time-minimal formulation in the sub-Lorentzian case is non-trivial, in contrast to the straightforward passage to the time-optimal formulation in the sub-Finsler case.} respectively, and writes Pontryagin’s maximum principle (PMP) in a left-invariant form for Lie groups. Geodesics may, in general, be singular extremals of PMP, so Section \ref{sec:singular_extremals} discusses their determination. Section \ref{sec:convex_trig} provides a brief introduction to convex trigonometry, which is then employed in Section \ref{sec:Heisenberg_Finsler} to obtain explicit formulas for Finsler geodesics on the Heisenberg group. For Lorentzian and sub-Lorentzian problems, the natural object is not a norm but an anti-norm, and geodesics can be expressed not via convex trigonometric functions generalizing $\cos$ and $\sin$, but via functions generalizing $\cosh$ and $\sinh$. Hence Section \ref{sec:antinorms_and_balls} discusses general properties of anti-norms, and Section \ref{sec:convex_trigonometry} constructs convex trigonometric functions extending the classical hyperbolic functions $\cosh$ and $\sinh$ to an arbitrary two-dimensional anti-norm. Section \ref{sec: light_time_like} proves Theorem \ref{thrm: light_space_like} on the separation of extremals in the sub-Lorentzian problem into lightlike and timelike ones. Applying this theorem and the functions $\mathrm{ch}\Omega, \mathrm{sh}\Omega$, Section \ref{sec:lobachevski} derives explicit formulas for extremals in the Lorentz problem on the Lobachevsky plane, while Section \ref{sec:3d_groups} does the same for the sub-Lorentzian problems on all three-dimensional unimodular Lie groups.

\section{Left-Invariant Sub-Finsler and Sub-Lorentz Problems}
\label{sec:general_statements}

We begin with a general formulation of the problems of left-invariant geodesics.  
Let $G$ be a real Lie group and let $\mathfrak{g} \simeq T_1G$ be its Lie algebra.  
A left-invariant sub-Finsler structure on $G$ is specified by a convex compact set\footnote{Here, as usual, $\mathrm{ri}\,\UU$ denotes the relative interior of the convex set~$\UU$.} $\UU \subset \mathfrak{g}$ such that $0 \in \mathrm{ri}\,\UU$.  
A left-invariant sub-Lorentz structure on $G$ is specified by a convex closed set $\UU \subset \mathfrak{g}$ such that $0 \notin \UU$ and, for all $\xi \in \UU$ and $ \lambda \ge 1$, one has $\lambda\xi \in \UU$ (in this case we say that $\UU$ satisfies the \textit{ray property}).  
In both cases the prefix \textit{sub} is used when the corresponding set lies in a proper subspace of $\mathfrak{g}$.  
Otherwise, when $\mathrm{int}\,U\ne\varnothing$, one simply speaks of a Finsler or Lorentz structure.

In the sub-Finsler case the Minkowski functional $\mu_{\UU}$,
\[
    \mu_\UU(\xi) = \inf\{\lambda>0 : \xi\in \lambda\UU\},
\]
defines an \emph{almost norm} on $\mathfrak{g}$ in the sense that  
$\mu_\UU$ satisfies all the properties of a norm except symmetry, $\mu_\UU(\xi)\ne\mu_\UU(-\xi)$, and finiteness (if $\mathrm{span}\,\UU\ne \mathfrak{g}$ then\footnote{We use the standard convention of convex analysis: $\sup\varnothing=-\infty$ and $\inf\varnothing=+\infty$.} $\mu_\UU(\xi)=+\infty$ for $\xi\notin\mathrm{span}\,\UU$).

It is well known that for a convex closed set $\UU$ the following holds: if $0\in\UU$, then $\xi\in\UU$ iff $\mu_\UU(\xi)\le 1$ (this property, in particular, guarantees a one-to-one correspondence between convex compact sets $\UU$ whose relative interior contains the origin and almost norms).  
In the sub-Lorentz case the detailed relationship between anti-norms and their unit balls has additional specifics, discussed in~\S\ref{sec:antinorms_and_balls}.

We now give the precise definition of an anti-norm (see~\cite{Protasov}).  
As usual, for a concave function $\nu:\R^n\to\R\sqcup\{-\infty\}$ we denote its domain by
\[
    \dom\nu=\{\xi\mid\nu(\xi)\ne-\infty\}.
\]
\begin{defn}
    \label{defn: antinorm}
    A function $\nu:\R^n\to\R\sqcup\{-\infty\}$ is called an \emph{anti-norm} if
    \begin{enumerate}
        \item $\nu$ is concave and closed\footnote{That is, its hypograph is closed, which is equivalent to upper semicontinuity.};
        \item $\nu$ is positively homogeneous, i.e.\ $\nu(\lambda\xi)=\lambda\nu(\xi)$ for all $\lambda\ge 0$ and $\xi\in\R^n$;
        \item $\dom\nu\ne\varnothing$ and $\nu(\xi)>0$ for every\footnote{Here $\mathrm{ri}$ denotes the relative interior of a set.} $\xi\in\mathrm{ri}\,\dom\nu$.
    \end{enumerate}
\end{defn}
\begin{defn}
    \label{defn: antinorm_ball}
    The \emph{unit ball} of an anti-norm $\nu$ is the set $U=\{\xi\in \R^n \mid \nu(\xi) \ge 1\}$.
\end{defn}

Next, we discuss how to construct an anti-norm from its unit ball.  
The unit ball $\UU$ of any anti-norm is always a convex, closed, unbounded set that does not contain the origin and satisfies the \textit{ray property}: for every $\lambda \ge 1$ one has $\lambda \UU \subset \UU$ (see~\S\ref{sec:antinorms_and_balls} for details).

In the sub-Lorentz case the functional $\nu_\UU$ that serves as the anti-norm with unit ball $\UU$ is defined slightly differently from the Minkowski functional in the sub-Finsler case.  
We start with an auxiliary functional $\mathring\nu_\UU$ defined analogously to the Minkowski functional:
\[
    \mathring\nu_\UU(\xi) = \sup\{\lambda>0 : \xi \in \lambda\UU\}.
\]
The functional $\mathring\nu_\UU$ is always concave by construction, but it need not be closed:

\begin{example}
    Let $\UU=\{x^2-y^2\ge 1\}\subset \R^2$.  
    Then $\mathring\nu_\UU(x,y)=x^2-y^2$ if $x^2-y^2>0$, and $\mathring\nu_\UU(x,y)=-\infty$ otherwise.  
    Hence $\mathring\nu_\UU$ is not closed; for instance, $\mathring\nu_\UU(0,0)=-\infty$.
\end{example}

Therefore, we define the functional $\nu_\UU$ that gives the anti-norm with unit ball $\UU$ as the closure of $\mathring\nu_\UU$:
\[
    \label{defn:antinorm}
    \nu_\UU = \mathrm{cl}\,\mathring\nu_\UU
    \quad \Longleftrightarrow\quad
    \nu_\UU(\xi) = \limsup_{\eta\to\xi} \mathring\nu_\UU(\eta).
\]
The functional $\nu_\UU$ defines an anti-norm on $\mathfrak{g}$.  
For what follows, the following property of $\UU$ is crucial: $\xi\in\UU$ iff $\nu_\UU(\xi)\ge 1$ (in particular, this guarantees the one-to-one correspondence $\UU\leftrightarrow\nu_\UU(\cdot)$).  
Note that this property holds far from universally for convex closed sets $\UU$ with $0\notin\UU$; namely, $\UU$ must satisfy the ray property mentioned above (see~\S\ref{sec:antinorms_and_balls} for details).

In both cases the almost norm and the anti-norm are transported by left translations $L_q(x)=q\cdot x$ over the whole group $G$, $q\in G$, and the length of an arbitrary Lipschitz\footnote{More generally, absolutely continuous; however, every such curve can be reparameterized to become Lipschitz without changing the value of the functional, so throughout the paper we seek optimal curves in the class of Lipschitz curves.} curve $q:[0,T]\to G$ is computed by
\footnote{The letters $f$ and $l$ are chosen to emphasize the difference between Finsler and Lorentz lengths.}
$$
    \ell_f(q) = \int_0^T \mu_{\UU}\bigl(dL_{q(t)}^{-1}\dot{q}(t)\bigr)\,dt,
    \qquad
    \ell_l(q) = \int_0^T \nu_{\UU}\bigl(dL_{q(t)}^{-1}\dot{q}(t)\bigr)\,dt.
$$

In the sub-Finsler case (as in the Riemannian case) one expects that curves connecting two given points cannot have arbitrarily small length.  
Hence the objects of interest are the \textit{shortest} curves, i.e.\ the solutions of
$$
    q(0) = q_0,\ q(T) = q_1,\ \ell_f(q) \to \min.
$$

In the sub-Lorentz case (as in the classical Minkowski space $\R^{3,1}$) one expects that curves connecting two given points cannot have arbitrarily large length\footnote{The interpretation from the viewpoint of special relativity is as follows: in a sub-Lorentz problem the length of a curve is precisely the proper time experienced by a traveller moving along that curve in space-time.}.  
Thus the objects of interest are the \textit{longest} curves, i.e.\ the solutions of
\[
    q(0) = q_0,\ q(T) = q_1,\ \ell_l(q) \to \max.
\]

In both cases, by left-invariance, we may assume $q_0=1$.  
Also, in both cases one may fix $T$, for instance $T=1$, but it will be more convenient for us to regard $T\ge 0$ as free.
```latex
\section{Problem Statements}
\label{sec:concrete_statements}

The paper describes extremals (i.e.\ solutions of Pontryagin’s Maximum Principle) for the following problems: 

\begin{itemize}
    \item the Finsler problem on the three–dimensional Heisenberg group $\mathbb{H}_3$ (see paragraph~\ref{sec:Heisenberg_Finsler});
    \item the Lorentz problem on the Lobachevsky plane (which can be described as the group of orientation-preserving affine transformations of the line), $\mathrm{Aff}_+\mathbb{R}$ (see paragraph~\ref{sec:lobachevski});
    \item sub-Lorentz problems on all three–dimensional unimodular Lie groups $SU(2)$, $SL(2)$, $SE(2)$, $SH(2)$, and $\mathbb{H}_3$ (see paragraph~\ref{sec:3d_groups}).
\end{itemize}

\subsection{The Finsler problem on the three–dimensional Heisenberg group}
\label{subsec: 3D Heisenberg}

Let $\mathbb{H}_3$ be the three-dimensional Heisenberg group, i.e.\ the set of matrices 
\[
    \begin{pmatrix}
        1 & x_1 & x_3\\
        0 & 1   & x_2\\
        0 & 0   & 1
    \end{pmatrix}
\]
with the usual matrix multiplication, and let $1$ denote the identity matrix.  
In this work we consider Finsler structures on $\mathbb{H}_3$ specified by a set $\UU$ that admits generalized spherical coordinates.  
Namely, let $\Omega \subset \mathbb{R}^2$ be a compact convex set containing $0$ in its interior, let $m>0$, $M>0$, and let $f:[-m,M]\to\mathbb{R}$ be a continuous concave function that is positive on $(-m,M)$.  
Set
\[
    \UU =
    \left\{
        \begin{pmatrix}
              0 & u_1 & u_3\\
              0 & 0   & u_2\\
              0 & 0   & 0
        \end{pmatrix}
        \ \Bigg|\ 
        u_3\in[-m,M],\ \exists \omega\in\Omega : (u_1,u_2)=f(u_3)\,\omega
    \right\}
    \subset T_1\mathbb{H}_3.
\]

Note that $f$ need not be strictly concave and need not vanish at the endpoints of the interval; for example, $f(v)=-|v|+1$ or $f\equiv\mathrm{const}$.  

It is easy to see that the set $\UU$ thus constructed is convex, compact, and contains $0$ in its interior (examples include ellipsoids centred at the origin with semi-axes along the coordinate axes, the octahedron centred at the origin, the unit balls of the norms $\|(x_1,x_2,x_3)\|^p=\sum_{i=1}^3|x_i|^p,\ p\ge1$, or a cylinder with base~$\Omega$ when $f\equiv\mathrm{const}$).

We pose the Finsler problem as in paragraph~\ref{sec:general_statements}:
\[
    q(0)=1,\ q(T)=q_1,\ 
    \ell_f(q)=\int_0^T \mu_{\UU}\bigl(dL_{q(t)}^{-1}\dot{q}(t)\bigr)\,dt \;\to\; \min,
\]
where $L_q(x)=q\cdot x$ is the left translation of $x\in\mathbb{H}_3$ by $q\in\mathbb{H}_3$, $\mu_{\UU}$ is the Minkowski functional of $\UU$, and $q(t)$ is a Lipschitz curve joining the identity to the point $q_1$.  

Explicit formulas for the geodesics in this problem are given in paragraph~\ref{sec:Heisenberg_Finsler}.

\subsection{The Lorentz problem on the Lobachevsky plane}

The Lobachevsky plane can be viewed as a Riemannian manifold obtained from a left-invariant Riemannian metric on the group $\mathrm{Aff}_+\R$ of orientation-preserving affine transformations of the line (see~\cite{Gribanova, Myrikova, ALS2021}).  
Building on this idea, we describe a Lorentz structure on the Lobachevsky plane.

Consider the group $G=\mathrm{Aff}_+\R$ of orientation-preserving affine transformations of the line.  
This is a two-dimensional Lie group generated by translations $x\mapsto a+x$ and dilations $x\mapsto bx$, $b>0$.  
Hence any element of the group can be represented as a pair $(a,b)$ with $b>0$, so the manifold $\mathrm{Aff}_+(\mathbb{R})$ is naturally identified with the upper half-plane.

If the Lie algebra $\mathfrak{g}=T_1G=\{(\xi,\eta)\}\simeq\R^2$ is equipped with the Euclidean norm $\|(\xi,\eta)\|^2=\xi^2+\eta^2$, the resulting left-invariant structure turns $\mathrm{Aff}_+\mathbb{R}$ into the Poincaré model of the Lobachevsky plane.  
The general Finsler problem on the Lobachevsky plane was solved in \cite{ALS2021}; the exact isoperimetric inequality in the Finsler case was completely resolved in~\cite{Myrikova}.  
The Lorentz problem on the Lobachevsky plane with a quadratic anti-norm was solved in~\cite{SachkovLobachevskiy}.  
In the present paper we find the extremals on $\mathrm{Aff}_+\mathbb{R}$ corresponding to an arbitrary Lorentz structure (see paragraph~\ref{sec:lobachevski}).  
We now formulate this problem.

Let $\UU \subset \mathfrak{g}$ be a convex, closed, unbounded set that does not contain the origin and satisfies the ray property.  
Then $\UU$ is the unit ball of some anti-norm $\nu_{\UU}$ on $\mathfrak{g}$ (see Definition~\ref{defn:antinorm}).  
Transporting this anti-norm by left translations to the tangent plane at every point of $G$ yields a Lorentz structure on $G$.

As stated in paragraph~\ref{sec:general_statements}, the Lorentz problem on this group is
\[
    \label{subl: affr}
    q(0)=1,\ q(T)=q_1,\ 
    \ell_l(q)=\int_0^T \nu_{\UU}\bigl(dL_{q(t)}^{-1}\dot{q}(t)\bigr)\,dt \;\to\; \max,
\]
where $L_q(x)=q\cdot x$ is the left translation of $x$ by $q$ in $\mathrm{Aff}_+\R$, $\nu_{\UU}$ is the anti-norm determined by the unit ball $\UU$ as in~\S\ref{defn:antinorm}, and $q(t)$ is a Lipschitz curve joining the identity to $q_1$.

\subsection{Sub-Lorentz problems on three-dimensional unimodular Lie groups}
\label{ssec:3d_lie_groups_porblem_statement}

Let $G$ be a three-dimensional unimodular Lie group and $\mathfrak{g}$ its Lie algebra.  
We consider all possible sub-Lorentz problems on such groups.  
Set $\Delta=\mathrm{span}\,\UU\subset\mathfrak{g}$.  
Since $\dim\mathfrak{g}=3$, we have $\dim\Delta\le3$.  
We classify the possible structures.

\smallskip\noindent
\textbf{Case $\boldsymbol{\dim\Delta=0}$.}  
Trivial, because the length of any non-constant curve equals $-\infty$.

\smallskip\noindent
\textbf{Case $\boldsymbol{\dim\Delta=1}$.}  
By the Picard existence and uniqueness theorem there exists a unique curve through $q_0=1$ tangent to the left translations of $\Delta$, namely the one-parameter subgroup $e^\Delta$, and motion along this curve is possible only in one direction.  
Since the length is independent of re-parameterization, there is no problem of choosing the \emph{best} curve from $q_0=1$ to $q_1$: either $q_1\notin e^\Delta$, in which case every curve from $q_0$ to $q_1$ has sub-Lorentz length $-\infty$, or $q_1\in e^\Delta$, in which case only the segment of $e^\Delta$ between $q_0$ and $q_1$ can have finite length, provided the motion is in the correct direction.

\smallskip\noindent
\textbf{Case $\boldsymbol{\dim\Delta=3}$.}  
Not studied here, as the paper deals with sub-Lorentz problems with a two-dimensional control set.

\smallskip\noindent
\textbf{Assume $\boldsymbol{\dim\Delta=2}$.}  
Two principal situations arise:

\begin{enumerate}
    \item[$\bullet$] $[\Delta,\Delta]\subset\Delta$.  
    Then $\tilde G=e^{\Delta}$ is a two-dimensional subgroup of $G$.  
    If $q_1\notin\tilde G$, every curve from $q_0=1$ to $q_1$ has length $-\infty$.  
    Thus we may assume $q_1\in\tilde G$.  
    Since $\dim\tilde G=2$, only two Lie algebras are possible:
    \begin{itemize}
        \item the commutative group $\Delta\simeq\R^2$ when $\dim[\Delta,\Delta]=0$ (the Lorentz geodesics in two-dimensional Minkowski space are well known);
        \item the algebra of $\mathrm{Aff}_+(\R)$ described above (geodesics found in paragraph~\ref{sec:lobachevski}).
    \end{itemize}

    \item[$\bullet$] $[\Delta,\Delta]\not\subset\Delta$.  
    Then $\Delta+[\Delta,\Delta]=\mathfrak{g}$, i.e.\ Hörmander’s condition holds.  
    We classify the possible positions of the plane $\Delta$ in $\mathfrak{g}$ with respect to the Lie algebra structure.  
    According to \cite{ALS2021, AgrachevBarilari}, one can choose a basis $f_1,f_2$ of $\Delta$ such that
    \begin{equation}
        \label{eq:3d_lie_brackets}
        [f_1,f_2]=f_3\notin\Delta,\qquad
        [f_3,f_1]=a\,f_2,\qquad
        [f_3,f_2]=b\,f_1,
    \end{equation}
    where $a,b\in\{0,\pm1\}$ and $a+b\ge0$.  
    More precisely,
    \begin{itemize}
        \item $a=b=0$: $\mathfrak{g}=\mathfrak{h}_3$;
        \item $a=1$, $b=0$: $\mathfrak{g}=\mathfrak{se}_2$;
        \item $a=0$, $b=1$: $\mathfrak{g}=\mathfrak{sh}_2$;
        \item $a=\pm1$, $b=1$: $\mathfrak{g}=\mathfrak{sl}_2$ (here $\Delta$ can be positioned in two essentially different ways);
        \item $a=1$, $b=-1$: $\mathfrak{g}=\mathfrak{su}_2$.
    \end{itemize}
\end{enumerate}

We again pose the sub-Lorentz problem as in paragraph~\ref{sec:general_statements}:
\[
    q(0)=1,\ q(T)=q_1,\ 
    \ell_l(q)=\int_0^T \nu_{\UU}\bigl(dL_{q(t)}^{-1}\dot{q}(t)\bigr)\,dt \;\to\; \max,
\]
where $L_q(x)=q\cdot x$ is the left translation of $x$ by $q$ in the chosen group, $\nu_{\UU}$ is the anti-norm determined by the unit ball $\UU$ as in paragraph~\ref{defn:antinorm}, and $q(t)$ is a Lipschitz curve joining the identity to $q_1$.  
Because of~\eqref{eq:3d_lie_brackets}, we have $dL_{q(t)}^{-1}\dot{q}(t)=u_1f_1+u_2f_2$ for some $u_1,u_2\in\R$, so the problem reduces to finding functions $(u_1,u_2)(t)$ that maximize the functional.  
It turns out that in all these cases the extremal equations on the Lie coalgebra can be integrated in a unified way (see paragraph~\ref{sec:3d_groups}).
```
\section{Transition to Minimum-Time and Maximum-Time Problems}
\label{sec:optimal_control_problem}

In both the sub-Finsler and sub-Lorentz settings the shortest, respectively longest, curves satisfy the Pontryagin Maximum Principle with Hamiltonians
\begin{equation}
\label{eq:Pontryagin_function_generalized_fins}
    \mathcal{H}_f(q,p,u)=\langle p, dL_{q}u\rangle+\lambda_0\mu(u)
    \;\longrightarrow\;\max_{u\in\R\UU},
\end{equation}
for the sub-Finsler problem, where $\mu$ is the almost norm with unit ball~$\UU$, and
\begin{equation}
\label{eq:Pontryagin_function_generalized_lorenz}
    \mathcal{H}_l(q,p,u)=\langle p, dL_{q}u\rangle-\lambda_0\nu(u)
    \;\longrightarrow\;\max_{u\in C},
\end{equation}
for the sub-Lorentz problem, where $\nu$ is the anti-norm,  
$C=\dom\nu=\mathrm{cl}\,\R_{+}\UU$, and $\UU$ is the unit ball of~$\nu$ (see~\S\ref{sec:general_statements}).  
Without loss of generality one may assume $\lambda_0\in\{0,-1\}$.

Recall that $\Delta=\mathrm{span}\,\UU$.

\medskip
\noindent\textbf{The sub-Finsler case.}
Let $U=\{u\mid\mu(u)\le1\}$ be the unit ball of~$\mu$.  
If a curve has finite sub-Finsler length, $\ell_f(q)<\infty$, then $\dot q(t)\in dL_{q(t)}\Delta$ for almost all~$t$.  
Indeed, should the set of times with $\dot q(t)\notin dL_{q(t)}\Delta$ have positive measure,  
$\mu_{\UU}\!\bigl((dL_{q(t)})^{-1}\dot q(t)\bigr)$ would equal $+\infty$ on a set of positive measure, forcing $\ell_f(q)=+\infty$.  
Consequently, when $q_0=1$ and $q_1$ can be joined by a finite-length curve, it suffices to minimise the length over curves $q(t)$ with $\dot q(t)\in dL_{q(t)}\Delta$ a.e.  
Since re-parameterisation does not change the length and\footnote{Here $\mathrm{ri}\,\UU$ denotes the relative interior of~$\UU$.} $0\in\mathrm{ri}\,\UU$, one may assume  
$\mu_{\UU}\!\bigl((dL_{q(t)})^{-1}\dot q(t)\bigr)\le1$, i.e.\ $\dot q(t)\in dL_{q(t)}\UU$ for $t\in[0,T]$ with some $T>0$ (when $q_0\neq q_1$).  
Under this parametrisation the length never exceeds the travel time~$T$; a shortest curve can be parametrised naturally by enforcing $\mu_{\UU}\!\bigl((dL_{q(t)})^{-1}\dot q(t)\bigr)=1$, in which case its length equals~$T$.

Thus the shortest curves solve the following \emph{minimum-time problem}:
\begin{equation}
\label{eq:subf_general_problem}
    \begin{gathered}
        T\;\longrightarrow\;\min,\\[4pt]
        q(0)=1,\quad q(T)=q_1,\\[4pt]
        \dot q(t)=dL_{q(t)}u(t),\qquad u(t)\in\UU.
    \end{gathered}
\end{equation}

Applying PMP, the Pontryagin function for~\eqref{eq:subf_general_problem} is
\begin{equation}
\label{eq:Pontryagin_function}
    \mathcal{H}(q,p,u)=\langle p, dL_{q}u\rangle
                     =\langle dL_{q}^{*}p, u\rangle,
\end{equation}
where $p\in T^{*}_{q}G$.  
If a pair $\hat q(t),\hat u(t)$ is optimal, there exists a Lipschitz curve  
$\hat p(t)\in T^{*}_{\hat q(t)}G$, $\hat p(t)\neq0$, such that  
$(\hat q,\hat p)$ obey Hamilton’s equations $\dot q=\mathcal{H}_{p}$,  
$\dot p=-\mathcal{H}_{q}$, and the control $\hat u(t)$ maximises
\[
    \langle dL_{q}^{*}p, u\rangle\;\longrightarrow\;\max_{u\in\UU}
\]
for almost all~$t$.

\medskip
\noindent\textbf{The sub-Lorentz case.}
Set $U=\{u\mid\nu(u)\ge1\}$, the unit ball of~$\nu$, and put $\R_{+}=\{\lambda\ge0\}$.  
Again we restrict attention to Lipschitz curves $q(t)$ with $\dot q(t)\in dL_{q(t)}\Delta$ a.e.  
If for some times $u(t)=dL_{q(t)}^{-1}\dot q(t)\notin C$, then $\nu_{\UU}(u(t))=-\infty$ there;  
if the set of such times has positive measure, the curve’s length is $-\infty$ and can be discarded (we seek the longest curves).  
Hence one may assume $u(t)\in C$ a.e.  
Unlike the sub-Finsler case, one cannot in general replace $C$ by $\R_{+}\UU$ because $\UU$ is unbounded.

We shall prove (see Theorem~\ref{thrm: light_space_like}) that every abnormal extremal ($\lambda_0=0$) is light-like, i.e.\ $\nu_{\UU}(u(t))=0$ a.e., whereas on each normal extremal ($\lambda_0=-1$) one has $u(t)\in\R_{+}\UU$ a.e.  
A related discussion on time-like vs.\ light-like longest curves and normality vs.\ abnormality can be found in~\cite{Podobryaev}.  
Note that $u=0$ is special: then both the velocity and the integrand vanish, so one can re-parameterise any curve by excising the intervals where $u(t)=0$ a.e., still connecting the same points and preserving $\ell_{l}$.  
Because $\nu_{\UU}(u(t))>0$ on the re-parameterised curve, a natural parametrisation with $\nu(u(t))=\text{const}$ exists.  
Thus, whenever the longest curve has positive length, it solves the following \emph{maximum-time problem}:
\begin{equation}
\label{eq:subl_general_problem}
    \begin{gathered}
        T\;\longrightarrow\;\max,\\[4pt]
        q(0)=1,\quad q(T)=q_1,\\[4pt]
        \dot q(t)=dL_{q(t)}u(t),\qquad u(t)\in\UU.
    \end{gathered}
\end{equation}
When analysing time-like extremals in paragraphs~\ref{sec:lobachevski} and~\ref{sec:3d_groups} we shall, in the maximum condition for $\mathcal{H}_{l}$ (with $\lambda_0=-1$), impose $\nu(u)=1$ a.e.

\medskip
\noindent\textbf{Summary.}
Time-like extremals ($\lambda_0=-1$) can be handled with either of the equivalent conditions
\[
    \langle dL_{q}^{*}p, u\rangle+\nu(u)\;\longrightarrow\;\max_{u\in C},
    \qquad
    \langle dL_{q}^{*}p, u\rangle\;\longrightarrow\;\max_{u\in\UU},
\]
whereas light-like extremals ($\lambda_0=0$) are governed by
\[
    \langle dL_{q}^{*}p, u\rangle\;\longrightarrow\;\max_{u\in C}.
\]

\bigskip

Finding extremals faces two main obstacles: the presence of singular extremals and the arbitrary shape of~$\partial U$.  
Paragraph~\ref{sec:singular_extremals} gives the definition of singular extremals and standard tools for dealing with them;  
paragraph~\ref{sec:convex_trig} summarizes the properties of convex trigonometric functions.  
Paragraph~\ref{sec:convex_trigonometry} develops the theory of convex hyperbolic functions, later used to find extremals in the sub-Lorentz problem on the motion group of the line (paragraph~\ref{sec:lobachevski}) and to integrate the vertical subsystem in the sub-Lorentz problems on unimodular Lie groups (paragraph~\ref{sec:3d_groups}).
\section{Singular Extremals}
\label{sec:singular_extremals}

The Pontryagin functions for the sub-Lorentz length–maximisation problem and for the sub-Finsler length–minimisation problem have the same general form, cf.~\eqref{eq:Pontryagin_function_generalized_fins},~\eqref{eq:Pontryagin_function_generalized_lorenz}, the only difference being the sign of the $\lambda_0$ term.

\begin{comment}
Since a maximisation problem is the same as a minimisation problem with the opposite sign in the functional, we have 
\begin{equation}
    \mathcal{H}_f(q,p,u)=\langle p, dL_{q}u\rangle+\lambda_0\mu(u)
    \;\to\;\max_{u\in\UU}
\end{equation}
for the sub-Finsler problem and 
\begin{equation}
    \mathcal{H}_l(q,p,u)=\langle p, dL_{q}u\rangle-\lambda_0\nu(u)
    \;\to\;\max_{u\in C=\mathrm{cl}\,\R_{+}\UU}
\end{equation}
for the sub-Lorentz problem.
\end{comment}

An extremal is called \emph{abnormal} if $\lambda_0=0$.  
Put $\lambda_0=-1$.  
As explained in \S~\ref{sec:optimal_control_problem}, in the sub-Finsler setting the extremals are those of the minimum-time problem; in \S~\ref{sec: light_time_like} we shall show that in the sub-Lorentz problems under consideration one may assume $u\in\UU$ in the normal case, and normal extremals coincide with the extremals of the maximum-time problem.  
Thus, in the normal case the maximum condition takes the unified form
\[
    \mathcal{H}(q,p,u)=\langle p, dL_{q}u\rangle
    \;\longrightarrow\;\max_{u\in\UU},
\]
both for the sub-Finsler and for the sub-Lorentz problems.

Because $\mathcal{H}$ is affine in $u\in\UU$, its maximum is attained on some \emph{face} of~$\UU$.  
Recall that a face of a convex set is its intersection with any supporting hyperplane.

In all the problems listed above, \emph{singular extremals} play a crucial role.  
Before solving them explicitly, we recall the definition of singular extremals\footnote{As always, by an \emph{extremal} we mean any solution of the PMP.}.

\begin{defn}
    Let $\mathbb{T}\subset[0,T]$ be a Lebesgue-measurable set of positive measure.  
    We say that an extremal is \emph{singular along the face $F$} of the convex set $\UU$, where $\dim F\ge1$, on~$\mathbb{T}$ if for almost every $t\in\mathbb{T}$ one has
    \[
        \argmax_{u\in\UU}\mathcal{H}(q,p,u)=F,
    \]
    irrespective of whether we are solving a sub-Finsler or a sub-Lorentz problem.
\end{defn}

The meaning is as follows.  
At each instant $t$ the optimal control is obtained by solving the finite-dimensional optimisation problem
\[
    \mathcal{H}=\langle dL_{q}^{*}p,u\rangle
    \;\longrightarrow\;\max_{u\in\UU}.
\]
Hence the covector $dL_{q}^{*}p$ must define a supporting half-space to~$\UU$; otherwise the maximum is unattainable.  
If the supporting half-space meets $\UU$ in a single point, the control $u$ at that instant is uniquely determined.  
If, however, the intersection is a face $F$ with $\dim F\ge1$, the control cannot be uniquely selected directly from the PMP.  
The remedy is that control values are defined up to a set of times of measure~$0$.  
Since a convex set contains at most countably many faces and, for each face $F$, the set of singular instants $\mathbb{T}_{F}$ is measurable, the control is uniquely determined on a set of full measure provided each $\mathbb{T}_{F}$ has measure~$0$.  

Conversely, if for some face $F$ the set $\mathbb{T}_{F}$ has positive measure, then $u(t)$ is \emph{not} directly specified by the PMP for $t\in\mathbb{T}_{F}$ and must be found from additional considerations that exploit the positive measure of~$\mathbb{T}_{F}$.
\section{Convex trigonometry}
\label{sec:convex_trig}

In this section we give a concise account of the apparatus of convex trigonometric functions (for details see~\cite{Lokut2019}). These functions will be used in the next section to derive explicit formulae for geodesics in the Heisenberg–group problem.

For convenience we use variables $(x,y)\in\mathbb{R}^2$ and $(p,q)\in\mathbb{R}^{2*}$, where $\mathbb{R}^{2*}$ is the space dual to $\mathbb{R}^2$. We assume that a non-empty set $\Omega\subset\mathbb{R}^2$ is convex, compact, and contains the origin in its interior. We now state the classical bipolar theorem.

\begin{theorem}[bipolar theorem]
    Let $\Omega\subset\mathbb{R}^2$ be a non-empty convex compact set containing the origin in its interior. Then its polar set
    \[
        \Omega^{\circ}=\bigl\{(p,q)\in\mathbb{R}^{2*}\mid px+qy\le1\ \forall(x,y)\in\Omega\bigr\}
    \]
    is itself a convex compact set containing the origin in its interior, and moreover $\Omega^{\circ\circ}=\Omega$.
\end{theorem}

By definition of the polar set, $(p,q)\in\Omega^{\circ}$ if and only if every point $(x,y)\in\Omega$ lies in the half-plane
\[
    \bigl\{(x,y)\in\mathbb{R}^2\mid px+qy-1\le0\bigr\}.
\]
The boundary $\partial\Omega$ can be parametrised by special functions $(\Ci,\ \Si)$.  Suppose first that $\theta\in[0,2S(\Omega))$, where $S(\Omega)$ denotes the area of $\Omega$.  Let $P\in\partial\Omega$ be such that the area of the sector of $\Omega$ between the rays $Ox$ and $\mathbb{R}_{+}(\cos_{\Omega}\theta,\sin_{\Omega}\theta)$ equals $\frac12\theta$.  By definition, $(\Ci,\ \Si)$ are the coordinates of~$P$.  If $\theta\notin[0,2S(\Omega))$, the functions $(\Ci,\ \Si)$ are extended periodically with period $2S(\Omega)$.  When $\Omega$ is the unit Euclidean disc centred at the origin, $(\Ci,\ \Si)$ coincide with the classical functions $(\cos\theta,\sin\theta)$.  The polar set is parametrised analogously by the pair $(\Co[\eta],\ \So[\eta])$.

Associate to every angle\footnote{We shall often call the argument of $\cos_{\Omega}$ and $\sin_{\Omega}$ an \emph{angle}, meaning the area of a sector of the indicated type.} $\theta$ (in general) a set of angles $\eta$ as follows.  Put $(x,y)=(\Ci,\ \Si)$ and consider the set of support lines to~$\Omega$ at $(x,y)$,
\[
    \bigl\{(p,q)\in\Omega^{\circ}\mid px+qy-1=0\bigr\}.
\]
Each point $(p,q)$ in this set lies on $\partial\Omega^{\circ}$ and therefore corresponds, as above, to a sector of area $\frac12\eta$.  The resulting set of $\eta$’s is denoted $\theta^{\circ}$.  Interchanging the roles of $\Omega$ and $\Omega^{\circ}$ produces, by analogy, for every angle $\eta$ a (possibly non-single) set of angles $\theta$, which we denote ${}^{\circ}\eta$.

Thus the angles $\theta$ and $\eta$ correspond to each other precisely when
\[
    \eta\in\theta^{\circ}
    \qquad\Longleftrightarrow\qquad
    \theta\in{}^{\circ}\eta.
\]
From the definition of the polar set and of the correspondence ${}^{\circ}$ it follows that
\begin{itemize}
    \item $\Co[\eta]\Ci+\So[\eta]\Si\le1$ for all $\theta,\eta$;
    \item $\Co[\eta]\Ci+\So[\eta]\Si=1$ if and only if $\theta\leftrightarrow\eta$.
\end{itemize}

The functions $\Ci$ and $\Si$ are Lipschitz and differentiable almost everywhere:
\[
    \Ci'(\theta)=-\So,
    \qquad
    \Si'(\theta)=\Co.
\]
Where the correspondence $\theta\mapsto\theta^{\circ}$ is not unique, the following intervals are bounded by the right and left derivatives:
\begin{itemize}
    \item $-\{\So[\eta]\mid\eta\in\theta^{\circ}\}$ for $\Ci$,
    \item $\{\Co[\eta]\mid\eta\in\theta^{\circ}\}$ for $\Si$.
\end{itemize}

For every point in the plane one can make an analogue of the polar change of variables,
\[
    (x,y)=A(\Ci,\Si)
    \quad\text{or}\quad
    (p,q)=B(\Co[\eta],\So[\eta]),
\]
for suitable $A,\theta$ (respectively $B,\eta$).  The Jacobian of these changes equals $A$ (respectively $B$).

If a point $(x,y)(t)=A(t)(\Ci,\Si)(t)$ moves in the plane so that $x(t),y(t)$ are absolutely continuous and $x^{2}(t)+y^{2}(t)\neq0$, then $A(t)$ and $\theta(t)$ are absolutely continuous as well and
\begin{equation}
\label{eq:dot_theta_dot_A}
    \dot{\theta}=
        \frac{\dot y\,\cos_{\Omega}\theta-\dot x\,\sin_{\Omega}\theta}{A},
    \qquad
    \dot{A}=
        \dot x\,\cos_{\Omega^{\circ}}\eta+
        \dot y\,\sin_{\Omega^{\circ}}\eta,
    \quad\text{where }\eta\in\theta^{\circ}.
\end{equation}

\section{Finsler problem on the Heisenberg group}
\label{sec:Heisenberg_Finsler}

In this section we obtain explicit formulae for extremals in Finsler problems on the Heisenberg group, assuming that the unit ball of the Finsler norm admits generalised spherical coordinates (see~\ref{subsec: 3D Heisenberg}).  With the aid of such coordinates extremals were found for a series of sub-Finsler problems on higher-dimensional Heisenberg groups, for instance for $\ell_p$ norms (see~\cite{Lokut2021}).  Throughout, we make essential use of the apparatus of convex trigonometry (see~\S\ref{sec:convex_trig}).

\medskip

To find geodesics we apply Pontryagin’s Maximum Principle to the time-optimal problem (see~\S\ref{sec:optimal_control_problem}).  For every Lipschitz trajectory $q(t)$ that minimises the functional\footnote{Such minimising trajectories exist by Filippov’s theorem~\cite{Filippov}.} there exists a Lipschitz curve $(q(t),p(t))\in T^*\mathbb{H}_3$ (called an \emph{extremal}) satisfying, for almost every~$t$,
\[
    \mathcal{H}(q,p,u):=\langle p,q\cdot u\rangle=\langle dL_q^{*}p,u\rangle,
\]
\begin{equation*}
    \begin{cases}
        \dot q = q\cdot u,\\[4pt]
        \dot p = -\dfrac{\partial\mathcal{H}}{\partial q}(q,p,u),\\[8pt]
        \displaystyle u\in\argmax_{u\in\UU}\mathcal{H}(q,p,u).
    \end{cases}
\end{equation*}

Let
\[
    q=\begin{pmatrix}
        1 & x_1 & x_3\\
        0 & 1 & x_2\\
        0 & 0 & 1
    \end{pmatrix},
    \qquad
    u=\begin{pmatrix}
        0 & u_1 & u_3\\
        0 & 0 & u_2\\
        0 & 0 & 0
    \end{pmatrix}.
\]
Then
\[
    \begin{cases}
        \dot x_1 = u_1,\\
        \dot x_2 = u_2,\\
        \dot x_3 = u_3 + x_1u_2.
    \end{cases}
\]

\medskip
\noindent\textbf{Vertical part of the PMP system.}
Because $\mathbb{H}_3$ is a Lie group, $T^*\mathbb{H}_3\simeq\mathbb{H}_3\times\mathfrak{h}_3^{*}$, where $\mathfrak{h}_3^{*}=T_1^{*}\mathbb{H}_3$ is the Lie coalgebra of $\mathbb{H}_3$.  The mapping
\[
    (q,p)\longmapsto(q,h),\qquad h=dL_q^{*}p,
\]
realises this diffeomorphism.  In these coordinates the equations for $h$ take the form
\begin{equation}
\label{eq:heisenberg_adjoint_system}
    \begin{cases}
        \dot h_1 = -h_3u_2,\\[4pt]
        \dot h_2 =  h_3u_1,\\[4pt]
        \dot h_3 = 0,\\[6pt]
        \displaystyle u\in\argmax_{u\in\UU}\bigl(u_1h_1+u_2h_2+u_3h_3\bigr).
    \end{cases}
\end{equation}

\begin{theorem}
    The first integrals of the vertical system~\eqref{eq:heisenberg_adjoint_system} are
    \[
        A=\mu_{\Omega^\circ}(h_1,h_2)=\mathrm{const}\ge0,
        \qquad
        h_3=\mathrm{const}\in\mathbb{R},
        \qquad
        A^{2}+h_{3}^{2}\neq0,
    \]
    where $\mu_{\Omega^\circ}$ is the Minkowski functional of the set $\Omega^\circ$.  According to their values, the PMP extremals issued at time $t=0$ from the identity element $(0,0,0)\in\mathbb{H}_3$ are as follows:
    \begin{enumerate}
        \item[\textup{1)}] If $A=0$, then $x_{1}(t)\equiv x_{2}(t)\equiv0$ and, writing\footnote{See the definition of $m$ and $M$ in~\S\ref{subsec: 3D Heisenberg}.} either $x_{3}(t)=-mt$ or $x_{3}(t)=Mt$.

        \item[\textup{2)}] If $A\neq0$ and $h_{3}=0$, then
        \[
            x_{1}(t)=W\int_{0}^{t}\cos_{\Omega}\theta(s)\,ds,\quad
            x_{2}(t)=W\int_{0}^{t}\sin_{\Omega}\theta(s)\,ds,
        \]
        \[
            x_{3}(t)=\int_{0}^{t}\!\bigl(u_{3}(s)+W x_{1}(s)\sin_{\Omega}\theta(s)\bigr)\,ds,
        \]
        where $\eta_{0}\in\mathbb{R}$ is a constant, $W=\max_{v\in[-m,M]}f(v)$, and $\theta(t)\in{}^{\circ}\eta_{0}$ and $u_{3}(t)\in\argmax_{v\in[-m,M]}f(v)$ are measurable functions.\footnote{For almost every $\eta_{0}$ the set ${}^{\circ}\eta_{0}$ is a singleton and the function $\theta(t)$ is constant (for every $\eta_{0}$ if $\partial\Omega$ has no corners).  The case of a non-singleton ${}^{\circ}\eta$ is discussed in Corollary~1.}

        \item[\textup{3)}] If $A\neq0$ and $h_{3}\neq0$, then
        \[
            \begin{aligned}
                x_{1}(t)=&\;\frac{A}{h_{3}}\bigl(\So[\eta(t)]-\So[\eta_{0}]\bigr),\\
                x_{2}(t)=&\;\frac{A}{h_{3}}\bigl(-\Co[\eta(t)]+\Co[\eta_{0}]\bigr),\\
                x_{3}(t)=&\;\frac{H}{h_{3}}\,t
                    -\frac12\frac{A^{2}}{h_{3}^{2}}\bigl(\eta(t)-\eta_{0}\bigr)
                    +\frac12\frac{A^{2}}{h_{3}^{2}}\Bigl(
                        2\cos_{\Omega^\circ}\eta(t)\sin_{\Omega^\circ}\eta_{0}
                        -\cos_{\Omega^\circ}\eta(t)\sin_{\Omega^\circ}\eta(t) \\
                    &\qquad\qquad\qquad
                        -\cos_{\Omega^\circ}\eta_{0}\sin_{\Omega^\circ}\eta_{0}
                    \Bigr),
            \end{aligned}
        \]
        where
        \[
            H=\max_{v\in[-m,M]}\bigl(Af(v)+h_{3}v\bigr)=\mathrm{const},\quad
            \eta(t)=\eta_{0}+\frac{Hh_{3}}{A^{2}}t-\frac{h_{3}^{2}}{A^{2}}\int_{0}^{t}u_{3}(s)\,ds,
        \]
        for some constant $\eta_{0}\in\mathbb{R}$ and a measurable function
        \[
            u_{3}(t)\in\argmax_{v\in[-m,M]}\bigl(Af(v)+h_{3}v\bigr).
        \]
    \end{enumerate}
    Conversely, every curve listed in~\textup{1)}, \textup{2)}, or~\textup{3)} lifts to a PMP extremal.
\end{theorem}

\begin{corol}
    If $f$ is strictly concave on $[-m,M]$, then $u_{3}$ is uniquely determined and constant (in cases~\textup{2} and~\textup{3}).  Hence the trajectories of type~\textup{3} are parametrised by the three parameters $\eta_{0}$, $h_{3}/A$, and~$t$.  If moreover, for some $\eta_{0}$, the set ${}^{\circ}\eta_{0}$ is a singleton, then the trajectories of type~\textup{2} are parametrised (in addition to $\eta_{0}$) by the single parameter~$t$.  When the set ${}^{\circ}\eta_{0}$ is not a singleton, the set
    $\{(\Ci,\Si)\mid\theta\in{}^{\circ}\eta_{0}\}$ is a segment
    $\bigl[\omega_{0},\omega_{1}\bigr]\subset\partial\Omega$, and the trajectories of type~\textup{2} have the form
    \[
        x_{3}=u_{3}t+W\int_{0}^{t}x_{1}(s)\sin_{\Omega}\theta(s)\,ds,
        \qquad
        (x_{1},x_{2})(t)\in Wt\bigl[\omega_{0},\omega_{1}\bigr],
    \]
    where the last integral can be expressed through $\beta(t)$, defined by
    $(x_{1},x_{2})(t)=Wt\bigl(\beta(t)\omega_{0}+(1-\beta(t))\omega_{1}\bigr)$:
    \[
        \begin{aligned}
            x_{3}(t)=\int_0^tu_3\,ds +\;&
            W^{2}\Bigl[
                \cos_{\Omega}\theta_{0}\sin_{\Omega}\theta_{0}\frac{\bigl(\beta(t)t\bigr)^{2}}{2}
                +\cos_{\Omega}\theta_{1}\sin_{\Omega}\theta_{1}\frac{\bigl(1-\beta(t)\bigr)^{2}t^{2}}{2}\\
                &\;+\cos_{\Omega}\theta_{1}\sin_{\Omega}\theta_{0}
                    \Bigl(\frac{t^{2}}{2}-\frac{\bigl(1-\beta(t)\bigr)^{2}t^{2}}{2}-\int_{0}^{t}\beta(s)s\,ds\Bigr)\\
                &\;+\cos_{\Omega}\theta_{0}\sin_{\Omega}\theta_{1}
                    \Bigl(\int_{0}^{t}\beta(s)s\,ds-\frac{\bigl(\beta(t)t\bigr)^{2}}{2}\Bigr)
            \Bigr].
        \end{aligned}
    \]
    Here $\theta_{0},\theta_{1}$ are the angles of the points $\omega_{0},\omega_{1}$ on~$\partial\Omega$, and $\beta(t)$ is an absolutely continuous function with $0\le\beta(t)\le1$ such that $0\leq (\beta(t)t)'\leq 1$.

    If $f$ is not strictly concave, the choice of $u_{3}(t)$ is in general not unique in case~\textup{2}, but is still unique for almost every ratio $A/h_{3}$ in case~\textup{3}.
\end{corol}

This corollary follows readily from Theorem~2.  For example, if the set $\Omega$ is strictly convex and the function~$f$ is strictly concave, then each initial value of the adjoint multiplier corresponds to a unique extremal.

\begin{proof}[Proof of Theorem 2]\ \\[-0.5\baselineskip]

As mentioned above, extremals are solutions of the PMP with the Pontryagin
function
\[
    \mathcal{H}(q,p,u):=\langle p,q\!\cdot\!u\rangle
        =\langle dL_q^{*}p,u\rangle .
\]
After the substitution $h=dL_q^{*}p$, the PMP equations reduce to
system~\eqref{eq:heisenberg_adjoint_system}.  To solve that system
explicitly we make the substitutions
\[
    u_1=f(u_3)\Ci[\theta],\qquad
    u_2=f(u_3)\Si[\theta],\qquad
    h_1=A\Co[\eta],\qquad
    h_2=A\So[\eta].
\]
Since $\dot h_3=0$, the component $h_3$ is constant.

The maximisation condition now reads
\[
    H=Af(u_3)\bigl(\Ci[\theta]\Co[\eta]+\Si[\theta]\So[\eta]\bigr)+h_3u_3
        \longrightarrow\max_{u_3,\theta}.
\]

Using~\eqref{eq:dot_theta_dot_A} and the definitions, we have $A(t)\ge0$.
From~\eqref{eq:heisenberg_adjoint_system} it follows that if
$A(t_0)\neq0$ at some time $t_0$, then in a neighbourhood of~$t_0$
\[
    \dot A=-h_3u_2\cos_{\Omega}\theta+h_3u_1\sin_{\Omega}\theta=0.
\]
Hence $A(t)$ is constant near any such $t_0$, and by continuity
$A(t)\equiv\text{const}\ge0$ for all~$t$.

\medskip
\noindent\emph{Case $A=0$ (extremals of type~1).}
Here we have a singular control: when $h_1=h_2=0$ the Hamiltonian
in~\eqref{eq:heisenberg_adjoint_system} attains its maximum for any
$u_1,u_2$.  If $h_1(\tau)=h_2(\tau)=0$ at some instant~$\tau$ then
$A(\tau)=0$, and since $A$ is constant we get $A(t)\equiv0$ and
$h_1(t)=h_2(t)=0$ for all~$t$.  By the PMP, $h_3(t)\neq0$, so the first
two equations of~\eqref{eq:heisenberg_adjoint_system} force
$u_1(t)=u_2(t)=0$.  The last condition in
\eqref{eq:heisenberg_adjoint_system} gives $u_3(t)\equiv M$ if $h_3>0$
and $u_3(t)\equiv-m$ if $h_3<0$.  Thus we obtain two singular extremals,
namely the two vertical rays.

\medskip
\noindent\emph{Let now $A>0$.}
From~\eqref{eq:dot_theta_dot_A} and
\eqref{eq:heisenberg_adjoint_system}
\[
    \dot\eta=\frac{h_3u_1\cos_{\Omega^\circ}\eta
                 +h_3u_2\sin_{\Omega^\circ}\eta}{A}
             =\frac{h_3}{A}\,f(u_3).
\]
If $f$ is strictly concave, the maximum of
$h_1u_1+h_2u_2+h_3u_3=Af(u_3)+h_3u_3$ with respect to $u_3$ is attained
at a unique point; since $A$ and $h_3$ are constant, $u_3$ is then
constant and $\eta$ is linear in~$t$.  If $f$ is not strictly concave,
the maximiser may be non-unique, but the PMP system can still be
integrated and explicit formulas obtained.

\medskip
\noindent\textbf{Extremals of type 2 ($h_3=0$).}
Because $\Ci[\theta]\Co[\eta]+\Si[\theta]\So[\eta]\le1$ with equality
only when $\theta\in{}^{\circ}\eta$, the Hamiltonian is strictly smaller
than $\max_{v\in[-m,M]}Af(v)$ unless $\theta\in{}^{\circ}\eta$.  Hence,
for $A>0$ and $h_3=0$, the maximum is attained iff
$\theta\in{}^{\circ}\eta$.  In this case $h_3\equiv0$, so
$h_1,h_2,\eta$ are constant: $\eta(t)\equiv\eta_0$.  The measurable
function $\theta(t)\in{}^{\circ}\eta_0$ may be chosen arbitrarily, while
$u_3(t)\in\argmax_{v\in[-m,M]}f(v)$.  Thus
\[
    \dot x_1=W\cos_{\Omega}\theta(t),\qquad
    \dot x_2=W\sin_{\Omega}\theta(t),\qquad
    \dot x_3=u_3(t)+x_1(t)u_2(t),
\]
where $W=\max_{v\in[-m,M]}f(v)$.  In other words, when $h_3=0$ the
control $u=(u_1,u_2,u_3)$ at each time $t$ can be taken from the fixed
face $F$ of the control set $U$ determined by the horizontal covector
$(\cos_{\Omega^\circ}\eta_0,\sin_{\Omega^\circ}\eta_0,0)$.

Because
\[
    \bigl\{(\Ci,\Si)\mid\theta\in{}^{\circ}\eta_0\bigr\}
        =[\omega_0,\omega_1]\subset\partial\Omega,\qquad
    \argmax_{[-m,M]}\bigl(Af(v)+cv\bigr)=[u_3^0,u_3^1],
\]
one can describe the reachable set explicitly.  At almost every~$t$
\[
    \bigl(\dot x_1,\dot x_2\bigr)
        =W\bigl(\alpha(t)\,\omega_0+(1-\alpha(t))\,\omega_1\bigr),\quad
    \dot x_3=\gamma(t)u_3^0+(1-\gamma(t))u_3^1
              +Wx_1(t)\sin_{\Omega}\theta(t),
\]
where $\alpha(t),\gamma(t)\in[0,1]$.  With the initial condition
$(x_1,x_2,x_3)(0)=(0,0,0)$ we have
\[
    (x_1,x_2)(t)\in Wt[\omega_0,\omega_1],\quad
    (x_1,x_2)(t)=Wt\bigl(\beta(t)\omega_0+(1-\beta(t))\omega_1\bigr),
\]
where $\displaystyle
    \beta(t)=\frac1t\int_0^t\alpha(s)\,ds\in[0,1]$.
The corresponding $x_3(t)$ is obtained by substituting these expressions
into the integral formula; see Corollary~1 for details.

\medskip
\noindent\textbf{Extremals of type 3 ($A\neq0$, $h_3\neq0$).}
Again the maximum is achieved iff $\theta\in{}^{\circ}\eta$.  Let
$\eta(0)=\eta_0\in\mathbb{R}$.  Then
\begin{align}
    x_1(t)
        &=\int_0^tf(u_3)\Ci\,ds
         =\int_0^t\frac{A}{h_3}\Ci\,\dot\eta\,ds
         =\frac{A}{h_3}\bigl(\So[\eta(t)]-\So[\eta_0]\bigr),
        \label{eq:Heisenberg_x1_en}\\[4pt]
    x_2(t)
        &=\int_0^tf(u_3)\Si\,ds
         =\int_0^t\frac{A}{h_3}\Si\,\dot\eta\,ds
         =\frac{A}{h_3}\bigl(-\Co[\eta(t)]+\Co[\eta_0]\bigr).
        \label{eq:Heisenberg_x2_en}
\end{align}
For $x_3$ we write
\[
    x_3(t)=\int_0^t\!\bigl(u_3+x_1u_2\bigr)\,ds
           =\!\int_0^t\!\Bigl(u_3
             +\frac{A}{h_3}\bigl(\So[\eta]-\So[\eta_0]\bigr)
              f(u_3)\Si\Bigr)\,ds.
\]
Since $\dot\eta=(h_3/A)f(u_3)$, we need the integral
\[
    I=\int\sin_{\Omega^\circ}\eta\,\sin_{\Omega}\theta\,d\eta
      =\tfrac12\bigl(\eta-\sin_{\Omega^\circ}\eta\cos_{\Omega^\circ}\eta\bigr)+c_1.
\]

Let $H=\displaystyle\max_{v\in[-m,M]}\bigl(Af(v)+h_3v\bigr)$.
Because $Af(u_3(t))+h_3u_3(t)=H$ for a.e.~$t$ and
$-\tfrac{h_3}{A}\in\partial f(u_3)$, we have
$f(u_3(t))=\dfrac{H}{A}-\dfrac{h_3}{A}\,u_3(t)$.  Hence
\[
    \eta(t)-\eta_0
        =\frac{h_3}{A}\int_0^tf(u_3)\,ds
        =\frac{h_3H}{A^{2}}\,t-\frac{h_3^{2}}{A^{2}}\int_0^tu_3\,ds,
\]
so that
\[
    \int_0^tu_3\,ds
        =\frac{H}{h_3}\,t-\frac{A^{2}}{h_3^{2}}\bigl(\eta(t)-\eta_0\bigr).
\]
Substituting, we obtain
\begin{equation}
\label{eq:Heisenberg_x3_en}
\begin{aligned}
    x_3(t)=\;
        &\frac{H}{h_3}\,t
        -\frac12\frac{A^{2}}{h_3^{2}}\bigl(\eta(t)-\eta_0\bigr)\\
        &+\frac12\frac{A^{2}}{h_3^{2}}\Bigl(
            2\cos_{\Omega^\circ}\eta(t)\sin_{\Omega^\circ}\eta_0
            -\cos_{\Omega^\circ}\eta(t)\sin_{\Omega^\circ}\eta(t)
            -\cos_{\Omega^\circ}\eta_0\sin_{\Omega^\circ}\eta_0
        \Bigr).
\end{aligned}
\end{equation}

Thus, when $A\neq0$ and $h_3\neq0$, formulas
\eqref{eq:Heisenberg_x1_en}–\eqref{eq:Heisenberg_x3_en} describe the
extremal parametrically by~$\eta$.  Equations
\eqref{eq:Heisenberg_x1_en}–\eqref{eq:Heisenberg_x2_en} contain
$\eta(t)$ but not $t$ explicitly, so $(x_1,x_2)$ moves along the
boundary of the polar set~$\Omega^\circ$, rotated by $\pi/2$ and scaled
by $A/h_3$, independently of the particular choice of~$\eta(t)$.  The
coordinate $x_3(t)$ depends explicitly on both $t$ and $\eta(t)$.  The
parameter $\eta(t)$ is determined by
\[
    \eta(t)=\eta_0+\int_0^t\frac{h_3}{A}f(u_3)\,ds,
\]
where $u_3(t)$ must maximise $Af(v)+h_3v$ over $v\in[-m,M]$.  If $f$ is
not strictly concave, the maximiser may be an interval, and any
measurable $u_3(t)$ with values in that interval yields an extremal.
When $f$ is strictly concave, the maximiser is unique, so
$u_3(t)\equiv\text{const}$ and $\eta(t)$ is linear in~$t$.  In
particular, for almost every (indeed, for all when $f$ is strictly
concave) ratio $A/h_3$, the parametrisation $\eta(t)$ is uniquely
determined by~$t$.

\medskip
\noindent\textbf{Reverse implication.}
Let arbitrary constants $A\ge0$, $h_3\in\mathbb{R}$ be given, not both
zero.  If $A=0$ and $\eta_0\in\mathbb{R}$, with an arbitrary measurable
$\theta(t)\in{}^{\circ}\eta_0$, then $h_1\equiv h_2\equiv0$ solves the
vertical part of the Hamiltonian system, while $u_3$ is fixed by the
sign of $h_3$ through $h_3u_3\to\max$.  If $h_3=0$ and $A>0$, put
$(h_1,h_2)=A\bigl(\Co[\eta_0],\So[\eta_0]\bigr)$; the vertical equations
are then satisfied.

Finally, let $A>0$, $h_3\neq0$, $\eta_0\in\mathbb{R}$, and choose any
measurable
\[
    u_3(t)\in\argmax_{v\in[-m,M]}\bigl(Af(v)+h_3v\bigr).
\]
Define $\eta(t)$ as in item~3 of the theorem and set
$(h_1,h_2)=A\bigl(\Co[\eta],\So[\eta]\bigr)$.  Since
$\dot\eta=(h_3/A)f(u_3)$,
\[
    \dot h_1=-A\dot\eta\,\Si=-h_3f(u_3)\Si,\qquad
    \dot h_2= A\dot\eta\,\Ci= h_3f(u_3)\Ci,
\]
with $\theta(t)\in{}^{\circ}\eta(t)$.  The controls corresponding to the
trajectory given in item~3 are
\[
    u_1=\dot x_1=f(u_3)\Ci,\qquad
    u_2=\dot x_2=f(u_3)\Si.
\]
Hence
\[
    \dot h_1=-h_3u_2,\qquad
    \dot h_2=h_3u_1,\qquad
    \max_{v\in U}\bigl(h_1v_1+h_2v_2+h_3v_3\bigr)
        =h_1u_1+h_2u_2+h_3u_3,
\]
the last equality holding by construction of $u_3(t)$ and the relation
$\theta(t)\in{}^{\circ}\eta(t)$.  Therefore every trajectory described
in the theorem is the projection of a solution to the Pontryagin maximum
principle.

\end{proof}

\section{Unit balls of anti-norms}
\label{sec:antinorms_and_balls}

The structure of the unit balls of ordinary norms is well known: every norm on a finite–dimensional space determines a unit ball, which is a (centrally symmetric) convex compact set containing the origin in its interior.  Conversely, every such set determines a unique norm (its Minkowski functional) for which it is the unit ball.  Anti-norms and the structure of their unit balls are not as widely known (though they have been studied by various authors, see, e.g.,~\cite{Protasov} and the references therein).  We shall use the notion of an \emph{anti-norm} from Definition~\ref{defn: antinorm} given earlier.

Denote by $U_\nu\subset\mathbb{R}^n$ the unit ball of an anti-norm~$\nu$:
\[
    U_\nu=\{\xi\in\mathbb{R}^n\mid\nu(\xi)\ge1\}.
\]

\begin{lemma}
\label{lm:antinorm_ball}
    Let\/ $\nu$ be an anti-norm on\/ $\mathbb{R}^n$.  Then
    \begin{enumerate}
        \item the unit ball $U_\nu$ is a non-empty closed convex set;
        \item $U_\nu$ enjoys the \emph{ray property:} $\lambda U_\nu\subset U_\nu$ for every $\lambda\ge1$;
        \item $0\notin U_\nu$.
    \end{enumerate}

    Conversely, if a set $U\subset\mathbb{R}^n$ satisfies the three properties above, then the function\/ $\nu_U:\mathbb{R}^n\to\mathbb{R}\cup\{-\infty\}$,
    \[
        \nu_U=\operatorname{cl}\,\mathring{\nu}_U,
        \qquad
        \mathring{\nu}_U(\xi)=\sup\{\lambda>0\mid\xi\in\lambda U\},
    \]
    is an anti-norm\footnote{As usual we adopt the convention $\sup\varnothing=-\infty$.}.

    Moreover, in these cases the transformations $\nu\mapsto U_\nu$ and $U\mapsto\nu_U$ are mutually inverse:
    \[
        \nu_{U_\nu}=\nu,
        \qquad
        U_{\nu_U}=U.
    \]
\end{lemma}

\begin{proof}
    The set $U_\nu$ is the image under the vertical projection
    $\mathbb{R}^n\times\mathbb{R}\to\mathbb{R}^n$, $(\xi,a)\mapsto\xi$, of the
    intersection of the hypograph of $\nu$ with the half-space $a\ge1$.
    Both the hypograph and the half-space are convex, hence their
    intersection is convex, and convexity is preserved under the linear
    map $(\xi,a)\mapsto\xi$.  Closedness of $U_\nu$ follows from the upper
    semicontinuity of $\nu$.  Since $\operatorname{dom}\nu\neq\varnothing$,
    we have $\operatorname{ri}\operatorname{dom}\nu\neq\varnothing$, whence
    there exists $\xi_0$ with $\nu(\xi_0)>0$.  By positive homogeneity this
    implies $U_\nu\neq\varnothing$.  Thus property~1 holds.  Property~2
    follows directly from positive homogeneity: if $\xi\in U_\nu$ and
    $\lambda\ge1$ then $\nu(\lambda\xi)=\lambda\nu(\xi)\ge1$.  For
    property~3 note that $\nu(0)=\limsup_{\xi\to0}\nu(\xi)\ge0$ (because
    $\nu$ is closed) and $\nu(\lambda0)=\lambda\nu(0)=0$ for all
    $\lambda>0$, so $\nu(0)=0$.

    \medskip
    Conversely, $\mathring{\nu}_U$ is clearly positively homogeneous.
    Its hypograph is
    \[
        \bigl\{(\xi,a)\;\bigm|\;\exists\lambda>0:\ \xi\in\lambda U,\ a\le\lambda\bigr\}
        =\bigcup_{\lambda>0}\lambda\bigl(U\times(-\infty,1]\bigr),
    \]
    a cone over the convex set $U\times(-\infty,1]$, hence itself convex;
    so $\mathring{\nu}_U$ is concave.  Its closure $\nu_U$
    is therefore a closed concave positively homogeneous function.  As
    $\operatorname{ri}\operatorname{dom}\nu_U
          =\operatorname{ri}\operatorname{dom}\mathring{\nu}_U$
    and $\mathring{\nu}_U>0$ on its domain, $\nu_U$ is an anti-norm.

    \medskip
    Finally, $\nu_{U_\nu}=\nu$ because on
    $\operatorname{ri}\operatorname{dom}\nu$ we have
    $\mathring{\nu}_{U_\nu}(\xi)=\nu(\xi)$ and the domains coincide.  For
    the other equality observe that
    $\xi\in U\Rightarrow\mathring{\nu}_U(\xi)\ge1\Rightarrow\nu_U(\xi)\ge1$,
    so $U\subset U_{\nu_U}$.  Conversely, if $\xi\in U_{\nu_U}$ then there
    exist $\xi_k\to\xi$ with $\mathring{\nu}_U(\xi_k)\to\lambda\ge1$, so
    $\xi_k/\lambda_k\in U$ and hence $\xi/\lambda\in U$; by the ray
    property $\xi\in U$.
\end{proof}

\section{Convex trigonometry for anti-norms}
\label{sec:convex_trigonometry}

In Section~\ref{sec:Heisenberg_Finsler} convex trigonometric functions
were used to obtain explicit geodesics in a Finsler problem on the
Heisenberg group.  For sub-Lorentzian problems it is convenient to
express geodesics through functions $\cosh_\Omega$ and $\sinh_\Omega$,
generalising the hyperbolic functions $\cosh$ and $\sinh$ from the
classical hyperbola to arbitrary unbounded convex sets.  The theory of
the convex trigonometric functions $\cos_\Omega$ and $\sin_\Omega$ was
developed in~\cite{Lokut2019} and exploited in many sub-Finsler problems
in~\cite{ALS2021}.  In this section we construct the apparatus of the
functions $\cosh_\Omega$ and $\sinh_\Omega$.

Recall the basic properties of $\cosh$ and $\sinh$:
\[
    \cosh^{2}\theta-\sinh^{2}\theta=1,\quad
    \frac{d}{d\theta}\cosh\theta=\sinh\theta,\quad
    \frac{d}{d\theta}\sinh\theta=\cosh\theta.
\]
Fixing $\theta_{0}>0$, the area bounded by the ray
$\mathbb{R}_{+}(1,0)=\{(x,0)\mid x>0\}$, the hyperbolic arc
$\{(\cosh\theta,\sinh\theta)\mid0<\theta<\theta_{0}\}$, and the segment
$\{\lambda(\cosh\theta_{0},\sinh\theta_{0})\mid0\le\lambda\le1\}$ equals
$\tfrac12\theta_{0}$.

We shall define $\cosh_\Omega,\sinh_\Omega$ so as to preserve these
properties as far as possible.

\subsection{Antipolar sets}

For convenience we keep the variables\footnote{Here $\R^{2*}$ denotes
the space of row–vectors dual to the vector space $\R^2$}
\((x,y)\in\R^{2}\) and \((p,q)\in\R^{2*}\).
By saying that a set \(\Omega\subset\R^{2}\) is
\emph{strictly separated from the origin} we mean that there exist
constants \(A,B\in\R\) such that \(Ax+By\ge1\) for every
\((x,y)\in\Omega\).

\begin{assmp}
\label{main_assumption}
Let $\Omega\subset\R^{2}$.
\begin{enumerate}
\item[(i)]
$\Omega$ is non-empty, convex and closed, does not contain the origin,
and satisfies the \emph{ray property}
\[
    \forall\lambda>1\qquad \lambda\Omega\subset\Omega.
\]
\end{enumerate}
\end{assmp}
\label{main assumptions i * **}
Throughout the paper we shall often assume, in addition to
\hyperref[main_assumption]{(i)}, one or both of the following related
properties.

\begin{assmp}
\label{dual_assmps}
\begin{enumerate}
\item[$(*)$]
No ray issued from the origin lies on\/ $\partial\Omega$:
\[
    \forall\lambda>1\qquad \lambda\Omega\cap\partial\Omega=\varnothing.
\]
\item[$(**)$]
The boundary\/ $\partial\Omega$ comes arbitrarily close to the two
boundary rays \(l_{0},l_{1}\) of the closed cone
\(C=\overline{\R_{+}\Omega}\):
\[
    \partial C=l_{0}\cup l_{1},\qquad
    \operatorname{dist}(l_{0},\Omega)=\operatorname{dist}(l_{1},\Omega)=0.
\]
\end{enumerate}
\end{assmp}

When a set $\Omega$ satisfies one of the above assumptions we shall say
that $\Omega$ \emph{has property} \hyperref[main_assumption]{(i)},
\hyperref[dual_assmps]{($*$)} or \hyperref[dual_assmps]{($**$)}
respectively.

If $\nu$ is an antinorm, its unit ball satisfies
\hyperref[main_assumption]{(i)} by
Lemma~\ref{lm:antinorm_ball}.  
It need not satisfy \hyperref[dual_assmps]{($*$)} or
\hyperref[dual_assmps]{($**$)}, although each of those can be reformulated
in terms of $\nu$.  For instance,
\hyperref[dual_assmps]{($*$)} is equivalent to
\(\nu(\xi)=0\) for every
\(\xi\in\dom\nu\setminus\operatorname{ri}\dom\nu\).

\begin{defn}
\label{defn: antipolar_set}
The \emph{antipolar} of a set\/ $\Omega\subset\R^{2}$ is
\[
    \Omega^{\diamond}=\bigl\{(p,q)\in\R^{2*}\mid
        px-qy\ge1 \;\; \forall(x,y)\in\Omega\bigr\}.
\]
\end{defn}

We write the term $-qy$ (rather than $+qy$) in the definition because of
the following example.  Let
\[
    \Omega_{2}
      =\bigl\{(x,y)\mid x^{2}-y^{2}\ge1,\;x>0\bigr\}
      =\bigl\{\lambda(\cosh\theta,\sinh\theta)\mid \lambda\ge1,\ \theta\in\R\bigr\}.
\]
Its antipolar in $(p,q)$–coordinates is
\[
    \Omega_{2}^{\diamond}
      =\bigl\{(p,q)\mid p^{2}-q^{2}\ge1,\;p>0\bigr\}
      =\bigl\{\lambda(\cosh\eta,\sinh\eta)\mid \lambda\ge1,\ \eta\in\R\bigr\},
\]
which coincides with $\Omega_{2}$ both under the Euclidean identification
$(x,y)\mapsto(x,y)$ and under the Lorentzian identification
$(x,y)\mapsto(x,-y)$.  
For $(x,y)=(\cosh\theta,\sinh\theta)\in\partial\Omega_{2}$ the supporting
covector is $(\cosh\theta,-\sinh\theta)$, i.e.\ $(p,q)=(\cosh\eta,\sinh\eta)$
with $\eta=\theta$.  Retaining the minus sign is therefore convenient.
Whenever we write $\omega^{\diamond}=(p,q)\in\Omega^{\diamond}$ we
always mean the coordinates given by
Definition~\ref{defn: antipolar_set}.

\medskip
We now state the main result of the subsection.

\begin{theorem}[Bipolar theorem]
\label{theorem:biantipolar}
Let\/ $\varnothing\ne\Omega\subset\R^{2}$.
\begin{enumerate}
\item
$\Omega$ is \emph{not} strictly separated from the origin
iff\/ $\Omega^{\diamond}=\varnothing$.
\item
If $\Omega$ \emph{is} strictly separated from the origin, then
$\Omega^{\diamond}$ satisfies
Assumption~\ref{main_assumption}\,\hyperref[main_assumption]{(i)} and
\[
    \Omega^{\diamond\diamond}
      =\operatorname{cl}\bigl(\operatorname{conv}
          \bigl(\textstyle\bigcup_{\lambda\ge1}\lambda\Omega\bigr)\bigr).
\]
\item
If\/ $\Omega$ fulfils
Assumption~\ref{main_assumption}\,\hyperref[main_assumption]{(i)}, then
$\Omega=\Omega^{\diamond\diamond}$.
\end{enumerate}
\end{theorem}

\begin{proof}
\emph{1.}
If $\Omega$ is strictly separated from the origin, so that
$Ax+By\ge1$ on $\Omega$, then $(A,-B)\in\Omega^{\diamond}$ and the latter
is non-empty.  Conversely,
$(p,q)\in\Omega^{\diamond}$ strictly separates $\Omega$ from~$0$.

\smallskip\noindent
\emph{3.}
By definition, $px-qy\ge1$ for all $(x,y)\in\Omega$ and each
$(p,q)\in\Omega^{\diamond}$, hence
\(\Omega\subset\Omega^{\diamond\diamond}\).
If $(x_{0},y_{0})\in\Omega^{\diamond\diamond}\setminus\Omega$, then the
segment $I=\{t(x_{0},y_{0})\mid0\le t\le1\}$ is disjoint from $\Omega$.
Because $I$ is compact and $\Omega$ is convex and closed, there exists a
supporting line $\{px-qy=1\}$ separating them.  This line represents a
point $(p,q)\in\Omega^{\diamond}$, yet
$px_{0}-qy_{0}<1$, contradicting
$(x_{0},y_{0})\in\Omega^{\diamond\diamond}$.

\smallskip\noindent
\emph{2.}
If $\Omega$ is strictly separated from $0$ then so is
$\operatorname{conv}\bigl(\bigcup_{\lambda\ge1}\lambda\Omega\bigr)$, and
the latter satisfies
Assumption~\ref{main_assumption}\,\hyperref[main_assumption]{(i)}.  
Since taking convex hulls, positive dilations and closure does not alter
the antipolar, part~3 yields the required identity, proving~2.
\end{proof}

Consequently, every set $\Omega$ with
\hyperref[main_assumption]{(i)} is itself the antipolar of its antipolar
$\Omega^{\diamond}$.

\medskip
We now illustrate antipolars with several explicit examples of sets
satisfying Assumption~\ref{main_assumption}.

\begin{example}
Let $P_{0}=(1,1)$, $P_{1}=(1,-1)$ and write
\[
    \R_{1+}=\{\lambda\in\R\mid \lambda\ge1\},\qquad
    [P_{0};P_{1}]
      =\{tP_{0}+(1-t)P_{1}\mid 0\le t\le1\}.
\]
Set
\[
    \Omega=\R_{1+}[P_{0};P_{1}]
           =\bigl\{(x,y)\mid |y|\le x,\;x\ge1\bigr\}.
\]
The corresponding antinorm $\nu_{\Omega}$ equals\/ $1$ on the segment
\([P_{0};P_{1}]\) and takes the value\/ $\lambda$ at the points of
$\lambda[P_{0};P_{1}]$, $\lambda\ge1$; namely
\[
    \nu_{\Omega}(x,y)=
    \begin{cases}
        x,& |y|\le x,\\[2pt]
        -\infty,& \text{otherwise}.
    \end{cases}
\]
A straightforward calculation gives
\[
    \Omega^{\diamond}
        =\bigl\{(p,q)\in\R^{2*}\mid |q|\le p-1\bigr\},
    \qquad
    \partial\Omega^{\diamond}
        =\R_{+}\bigl(P_{0}\cup P_{1}\bigr)+(1,0).
\]
The antinorm associated with $\Omega^{\diamond}$ is
\[
    \nu_{\Omega^{\diamond}}(p,q)=
    \begin{cases}
        p-|q|,& |q|\le p,\\[2pt]
        -\infty,& \text{otherwise}.
    \end{cases}
\]

Note that here $\Omega$ verifies property
\hyperref[dual_assmps]{($**$)} but not
\hyperref[dual_assmps]{($*$)}, whereas $\Omega^{\diamond}$ enjoys
\hyperref[dual_assmps]{($*$)} but fails
\hyperref[dual_assmps]{($**$)}.  Thus the two properties are dual to
each other.
\end{example}

\begin{example}[Antipolar of the right half of the $\alpha$-hyperbola]
\label{exmpl:omega_alpha}
On $\R^{2}$, $\alpha>1$, consider the antinorm
\[
    \|(x,y)\|_{\alpha}=|x|^{\alpha}-|y|^{\alpha}
    \quad\text{for }|y|\le x,
\]
finite exactly on the cone $\{|y|\le x\}$.  Its unit ball,
the \emph{$\alpha$-hyperbola}, is
\[
    \Omega_{\alpha}
      =\bigl\{(x,y)\in\R^{2}\mid
          |x|^{\alpha}-|y|^{\alpha}\ge1,\;x>0\bigr\}.
\]
The set $\Omega_{\alpha}$ is strictly convex and the distance from
$\partial\Omega_{\alpha}$ to the asymptotes $y=\pm x$ is zero, hence
$\Omega_{\alpha}$ satisfies both Assumptions~\ref{main_assumption} and
\ref{dual_assmps}.  Solving the minimisation problem
\(\min_{(x,y)\in\Omega_{\alpha}}(px-qy)\) by Lagrange multipliers one
finds that, when \(\alpha^{-1}+\beta^{-1}=1\), the antipolar is again an
$\beta$-hyperbola:
\[
    \Omega_{\alpha}^{\diamond}=\Omega_{\beta}.
\]
In particular, for $\alpha=2$ we have
\(\Omega_{2}^{\diamond}=\Omega_{2}\).

Observe that in this example \(\Omega=\Omega_{\alpha}\) and
\(\Omega^{\diamond}=\Omega_{\beta}\) satisfy all three properties
\hyperref[main_assumption]{(i)},
\hyperref[dual_assmps]{($*$)} and
\hyperref[dual_assmps]{($**$)}.  This leads to the following duality
result.
\end{example}

\begin{theorem}[Dual properties]
\label{theorm:dual_properties}
Let\/ $\Omega\subset\R^{2}$ satisfy
Assumption~\ref{main_assumption}\,\hyperref[main_assumption]{(i)}.
\begin{enumerate}
\item
$\Omega$ has property \hyperref[dual_assmps]{($*$)}  
iff\/ $\Omega^{\diamond}$ has property \hyperref[dual_assmps]{($**$)}.
\item
$\Omega^{\diamond}$ has property \hyperref[dual_assmps]{($*$)}
iff\/ $\Omega$ has property \hyperref[dual_assmps]{($**$)}.
\end{enumerate}
\end{theorem}

Before proving the theorem we need lemma. Everywhere
below we assume that $\Omega$ possesses
property~\hyperref[main_assumption]{(i)}.

\medskip
To decide whether a ray lies in the interior or on the boundary of
$\Omega$ (or of $\Omega^{\diamond}$) we use the next lemma.

\begin{lemma}
\label{lemma:inner_polar}
Let\/ $\Omega\subset\R^{2}$ satisfy
Assumption~\ref{main_assumption}.
\begin{enumerate}
\item
If\/ $\omega\in\operatorname{int}\Omega$,
then $\lambda\omega\in\operatorname{int}\Omega$ for all $\lambda>1$.
\item
If\/ $\omega\in\partial\Omega$, then either
$\lambda\omega\in\operatorname{int}\Omega$
for all $\lambda>1$ or
$\lambda\omega\in\partial\Omega$ for all $\lambda>1$.
\end{enumerate}
\end{lemma}

\begin{proof}
Let $\omega\in\mathrm{int}\,\Omega$.  Denote by $U$ a neighbourhood of $\omega$ with $U\subset\Omega$.  Because, by the hypothesis of the lemma, $\Omega$ satisfies Assumption~\ref{main_assumption}, for every $\lambda>1$ we have $\lambda U\subset\Omega$.  The set $\lambda U$ is open and is a neighbourhood of the point $\lambda\omega$, hence $\lambda\omega\in\mathrm{int}\,\Omega$.  This proves item~1.

Now let $\omega\in\partial\Omega$.  Suppose that the ray $\{\lambda\omega\mid\lambda>1\}$ is not entirely contained in $\partial\Omega$, i.e.\ there exists $\lambda>1$ such that $\lambda\omega\in\mathrm{int}\,\Omega$.  By the previous part of the lemma, for every $\mu>\lambda$ the point $\mu\omega$ lies in $\mathrm{int}\,\Omega$.  We show that the same is true for the points $\mu\omega$ with $1<\mu<\lambda$.  Let $U\subset\Omega$ be a neighbourhood of $\lambda\omega$.  Because $\Omega$ is convex, the union of all segments joining $\omega$ to points of $U$ is contained in $\Omega$.  The interior of that union is a neighbourhood of the set $\{\mu\omega\mid 1<\mu<\lambda\}$, which proves item~2.
\end{proof}

Note that if the non-empty set $\Omega$ is strictly separated from the origin, then by the biantipolar theorem~\ref{theorem:biantipolar} the set $\Omega^\diamond$ possesses property~\hyperref[main assumptions i * **]{$(i)$} and therefore also satisfies the assumptions of the above lemma.

We now give a criterion for a point $\omega^\diamond\in\Omega^\diamond$ to belong to the interior of $\Omega^\diamond$.
\begin{lemma}
\label{lemma:inner_polar_criteria}
Let $\gamma\colon(-1,1)\to\mathbb{R}^{2*}$ be a continuously differentiable curve such that
\[
\gamma(0)=\omega^\diamond\in\Omega^\diamond,
\qquad
\dot{\gamma}(0)\notin\{\lambda\omega^\diamond\mid \lambda\in\mathbb{R}\}.
\]
Assume that there exists $\delta>0$ such that for every $\alpha\in(0,\delta)$
\begin{enumerate}
\item[---] $(1-\alpha)\omega^\diamond\in\Omega^\diamond,$
\item[---] $\gamma(-\alpha)\in\Omega^\diamond,\quad\gamma(\alpha)\in\Omega^\diamond$,
\end{enumerate}
then $\omega^\diamond\in\mathrm{int}\,\Omega^\diamond$.  Conversely, if there is a curve $\gamma$ of the above type for which no such $\delta>0$ exists, then the point $\omega^\diamond$ lies on the boundary of the antipolar set.
\end{lemma}
\begin{proof}
($\Rightarrow$) For a $C^{1}$-curve and sufficiently small $\alpha>0$ the points $\gamma(\alpha)$ and $\gamma(-\alpha)$ lie on different sides of the line $\{\lambda\omega^\diamond\mid\lambda\in\mathbb{R}\}$ (and off the line itself).  By convexity of $\Omega^\diamond$,
\[
\mathrm{conv}\bigl(\gamma(-\alpha),\gamma(\alpha),(1-\alpha)\omega^\diamond,(1+\alpha)\omega^\diamond\bigr)\subset\Omega^\diamond,
\]
so $\omega^\diamond\in\mathrm{int}\,\Omega^\diamond$.

($\Leftarrow$) If no such $\delta>0$ exists, then every neighbourhood of $\omega^\diamond$ contains points outside $\Omega^\diamond$; hence $\omega^\diamond\in\partial\Omega^\diamond$.
\end{proof}

\medskip
We now construct a curve $\gamma$ convenient for the proof of the theorem and satisfying the conditions of the preceding lemma.  
Let $L(\omega^\diamond)$ be the line in the plane $\mathbb{R}^2$ determined by the element $\omega^\diamond=(p,q)\in\Omega^\diamond$, namely
\[
L(\omega^\diamond)=\{(x,y)\in\mathbb{R}^2\mid px-qy=1\}.
\]
Consider the affine transformation which is the rotation by an angle $\alpha\in\mathbb{R}$ in the positive direction about the point $d$ of $L(\omega^\diamond)$ closest to the origin.  This transformation sends lines to lines, and for $|\alpha|<\frac{\pi}{2}$ the parameters of the image line $(p,q)(\alpha)$ become the parameters $\omega^\diamond(\alpha)=(p,q)(\alpha)$ of the image of $L(\omega^\diamond)$, which does not pass through the origin.  Note that, in general, the point $\omega^\diamond(\alpha)$ does not belong to $\Omega^\diamond$ when $\alpha\neq0$.  Thus we obtain the curve $\gamma(\alpha)=\omega^\diamond(\alpha)$, $\alpha\in(-\pi/2,\pi/2)$.
\begin{lemma}
\label{lemma:inner_polar_gamma}
The curve $\gamma(\alpha)$ satisfies the requirements of Lemma~\ref{lemma:inner_polar_criteria}, namely
\[
\gamma(0)=\omega^\diamond,\qquad
\dot{\gamma}(0)\notin\{\lambda\omega^\diamond\mid \lambda\in\mathbb{R}\}.
\]
\end{lemma}
\begin{proof}
Write the parameters $(p,q)(\alpha)$ using the complex variable $p-iq$, noting that
$d=(d_1,d_2)=\frac{1}{p^{2}+q^{2}}(p,-q)$:
%%$$(p, -q)(\alpha) = \frac{1}{(p, -q)e^{-i\alpha}d}(p, -q)e^{-i\alpha}.$$
%%In complex form, and taking into account that $d = (d_1, d_2) = \frac{1}{p^2 + q^2}(p, -q)$ :
\[
\begin{aligned}
 (p-iq)(\alpha) &=
 \frac{1}{\operatorname{Re}\bigl((p+iq)e^{-i\alpha}(d_1+id_2)\bigr)}(p-iq)e^{i\alpha}\\
 &=\frac{1}{\cos\alpha}(p-iq)e^{i\alpha}
 =(1+i\tan\alpha)(p-iq).
\end{aligned}
\]
Differentiating with respect to $\alpha$,
\[
\frac{d}{d\alpha}(p-iq)=i\frac{1}{\cos^{2}\alpha}(p-iq)=\frac{1}{\cos^{2}\alpha}(q+ip).
\]
Hence $\dot{\gamma}(0)\not\parallel(p-iq)$, i.e.\ $\dot{\gamma}(0)\notin\{\lambda\omega^\diamond\mid \lambda\in\mathbb{R}\}$.
\end{proof}

\begin{lemma}
\label{lemma:min_distance}
There exists a unique point $\omega\in\Omega$ at which the minimum distance from $\Omega$ to the origin is attained.
\end{lemma}
\begin{proof}
Existence follows from the closedness of $\Omega$.  Suppose there are two such points.  Because $\Omega$ is convex, the segment joining them is contained in $\Omega$.  Any interior point of that segment is strictly closer to the origin, contradicting minimality.
\end{proof}

The next lemma will be needed for a convenient interpretation of property \hyperref[main assumptions i * **]{$(**)$}.
\begin{lemma}
\label{lemma:**_not_parallel}
Assume that $\Omega$ satisfies property \hyperref[main assumptions i * **]{$(**)$}.  
Then every line of the form 
\[
L(\omega^\diamond)=\{(x,y)\in\mathbb{R}^2\mid px-qy=1\}, 
\qquad
\omega^\diamond=(p,q)\in\Omega^\diamond,
\]
cannot be parallel to either of the boundary rays $l_0,l_1$ of the cone 
\[ 
C=\mathrm{cl}\bigl(\mathbb{R}_{+}\Omega\bigr).
\]
The assertion remains true even when the rays $l_0$ and $l_1$ coincide, i.e.\ when $\dim\Omega=1$.
\end{lemma}

\begin{proof}
Suppose that for some $(p,q)\in\Omega^\diamond$ the corresponding line  
$L=L(p,q)$ (which strictly separates $\Omega$ from the origin) is parallel to the ray $l_0$.  
Then $\Omega$ cannot lie between the line $L_0$ that contains $l_0$ and the line $L$, because $\Omega$ is contained in that half-plane determined by $L$ which does not contain the origin—and therefore does not contain $L_0$ either.  
Consequently, $\Omega$ lies entirely on one side of the two parallel lines $L_0$ and $L$.  
Hence 
\[
\operatorname{dist}\bigl(\Omega,l_0\bigr)
  \,\ge\,
\operatorname{dist}\bigl(\Omega,L_0\bigr)
  \,\ge\,
\operatorname{dist}\bigl(L,L_0\bigr)
  \;>\;0,
\]
which contradicts property \hyperref[main assumptions i * **]{$(**)$}.

If on the other hand every point $\omega\in\Omega$ belongs to the ray $l_0$, then no line parallel to that ray can, by definition, strictly separate the subset $\Omega$ of this ray from the origin.
\end{proof}

\noindent
The converse is also true: if the set $\Omega$ fails to satisfy property \hyperref[main assumptions i * **]{$(**)$}, then 
\[
\operatorname{dist}\bigl(l_k,\Omega\bigr)\neq 0
\quad\text{for some }k\in\{0,1\}.
\]
Translate that ray $l_k$ towards $\Omega$ by (say) one half of this distance in the direction perpendicular to $l_k$, and extend it to a full line.  
The resulting line is defined by an element of the antipolar set $\Omega^\diamond$ and, by construction, is parallel to the boundary ray of the cone~$C$.  

\medskip

\begin{proof}[Proof of Theorem~4]
The theorem is obvious in the degenerate cases $\dim\Omega=1$ or $\dim\Omega^\diamond=1$.  
Indeed, if $\Omega=\{(x,0)\mid x\ge 1\}$, then by definition 
\[
\Omega^\diamond=\{(p,q)\in\mathbb{R}^{2*}\mid p\ge 1\}.
\]
The statement of the theorem clearly holds for $\Omega$, and likewise after replacing $\Omega$ by $\Omega^\diamond$.  
Moreover, under a rotation of $\Omega$ about the origin, its antipolar set $\Omega^\diamond$ is rotated by the same angle, and if $\Omega$ is dilated by a factor $\lambda>0$, the set $\Omega^\diamond$ is contracted by the same factor~$\lambda$.  
Hence the theorem is settled for all degenerate cases.  
From now on we assume that $\mathrm{int}\,\Omega\neq\varnothing$.

\smallskip
Statements~1 and~2 of the theorem are equivalent: by the biantipolar theorem~\ref{theorem:biantipolar} the set $\Omega^\diamond$ also enjoys properties \hyperref[main assumptions i * **]{$(i)$}, and moreover $(\Omega^\diamond)^\diamond=\Omega$.

\medskip
We turn to the substantive part of the proof.  
We prove sufficiency in item~2, namely that property \hyperref[main assumptions i * **]{$(*)$} for $\Omega^\diamond$ implies property \hyperref[main assumptions i * **]{$(**)$} for $\Omega$.  
Arguing by contradiction, suppose that the distance from one of the boundary rays, say $l_0$, of the cone
\[
C=\mathrm{cl}\bigl(\mathbb{R}_{+}\Omega\bigr)
\]
to the set $\Omega$ is non-zero.  
Choose $0<\delta<1$ and translate the ray $l_0$ in the direction perpendicular to $l_0$ towards $\Omega$ by the amount $(1-\delta)$ of that distance.  
Denote this translation vector by $d\in\mathbb{R}^2$, and call the translated ray $l=l_0+d$.  

Because $l_0$ is a boundary ray of~$C$, there exists a sequence of points $\omega_k\in\Omega$ such that
\[
e_k=\frac{\omega_k}{|\omega_k|}\;\longrightarrow\; e_0,
\]
where $e_0$ is the direction of $l_0$.  
Write $\varphi_k$ for the angle between $e_k$ and $e_0$; then $\varphi_k\to 0$.  
The vertex of $l_0$ is the origin, so the vertex of $l$ is the point~$d$.  
Let $\alpha_k$ be the angle between the segment joining $\omega_k$ with $d$ and the ray~$l$.

\begin{lemma}
\label{lemma:small_angles}
\(
\alpha_k\longrightarrow 0 \quad\text{as }k\to\infty.
\)
\end{lemma}

\begin{proof}
We first show that $|\omega_k|\to\infty$.  
Otherwise, some subsequence $\omega_{k_i}$ would converge to a point $\omega\in\Omega$.  
Then $\varphi_{k_i}\to 0$ would imply $\omega\in l_0$, which is impossible because by assumption the distance from $l_0$ to $\Omega$ is positive.  
Hence $|\omega_k|\to\infty$.

Now compute the limit of $\cos\alpha_k$:
\[
\cos\alpha_k
  =\Bigl\langle e_0,\frac{d-\omega_k}{|d-\omega_k|}\Bigr\rangle
  =o(1)+\Bigl\langle e_0,\frac{\omega_k}{|\omega_k|}\Bigr\rangle
  =o(1)+\langle e_0,e_k\rangle.
\]
Since $\langle e_0,e_k\rangle=\cos\varphi_k$, we obtain
\[
\lim_{k\to\infty}\cos\alpha_k
  =\lim_{k\to\infty}\cos\varphi_k
  =1,
\]
and therefore $\alpha_k\to 0$.
\end{proof}

Let the line $L_0$ containing the ray $l_0$ be given by the equation $px-qy=0$, and let the line $L$ containing the ray $l=l_0+d$ be given by
\[
\lambda(px-qy)=1,
\qquad
\lambda(p,q)\in\Omega^\diamond,
\quad
\lambda=\lambda(\delta).
\]
Rotating the line $L$ through the angle $\alpha_k$ about the point $d$, the rotated line meets the set $\Omega$ at the point $\omega_k$ (by the very definition of $\omega_k$), and by the lemma we have $\alpha_k\to0$ as $k\to\infty$.  Hence an arbitrarily small rotation of $L$ causes it to intersect~$\Omega$, so $\lambda(p,q)\in\partial\Omega^\diamond$ by Lemmas~\ref{lemma:inner_polar_criteria} and~\ref{lemma:inner_polar_gamma}.  Since this is true for every $0<\delta<1$, Lemma~\ref{lemma:inner_polar} implies that all points of the form $\lambda(p,q)$ with $\lambda>1$ lie on the boundary $\partial\Omega^\diamond$, because $\Omega^\diamond$ satisfies Assumption~\ref{main_assumption} by Theorem~\ref{theorem:biantipolar}.  Thus $\Omega^\diamond$ fails to possess property \hyperref[main assumptions i * **]{$(*)$}—a contradiction.

\medskip
We now prove the converse part of item~2.  
Assume that $\Omega$ satisfies property \hyperref[main assumptions i * **]{$(**)$}.  We show that for every $(p,q)\in\Omega^\diamond$ all points of the ray $\{\lambda(p,q)\mid\lambda>1\}$ lie in the interior of the antipolar set $\Omega^\diamond$.

By Lemma~\ref{lemma:inner_polar} it suffices to find some $\lambda'>1$ such that the point $\lambda'(p,q)$ lies in $\mathrm{int}\,\Omega^\diamond$.  Using Lemma~\ref{lemma:inner_polar_gamma}, construct a curve~$\gamma$ passing through $\lambda'(p,q)$ and show that for sufficiently small~$\alpha$ we have $\gamma(\alpha)\in\Omega^\diamond$.  Geometrically, this means that for small rotations about the point $d(\lambda')$ nearest to the origin on the line $L(\lambda' p,\lambda' q)$, the rotated line still separates $\Omega$ from the origin.

For convenience, let $I$ be the segment joining the origin to the unique point of $\Omega$ closest to~$0$ (uniqueness follows from Lemma~\ref{lemma:min_distance}).  
By Lemma~\ref{lemma:**_not_parallel}, for every $\lambda>1$ the line $L(\lambda p,\lambda q)$ is not parallel to the boundary rays $l_0,l_1$.  By definition of separability this line intersects the segment~$I$ (hence meets the cone $C$ in interior points) but does not intersect $\Omega$.  Consequently it meets both rays $l_0$ and $l_1$.  Rotating $L(\lambda p,\lambda q)$ about $d(\lambda)$ through a sufficiently small angle~$\alpha$ does not destroy these intersections; denote the intersection points by $P_0(\lambda,\alpha)$ and $P_1(\lambda,\alpha)$.  Because every point of $L(\lambda p,\lambda q)$ is $\lambda$ times closer to the origin than the corresponding point of $L(p,q)$, we have
\[
P_k(\lambda,\alpha)=\frac1\lambda\,P_k(1,\alpha).
\]
The points $P_k(\lambda,\alpha)$ depend continuously on $(\lambda,\alpha)$ and satisfy $P_k(\lambda,\alpha)\to 0$ as $\lambda\to\infty$.  Hence there exists $\lambda'>1$ such that $P_0(\lambda',0)$ and $P_1(\lambda',0)$ lie inside the open ball of radius $\tfrac12|I|$ centred at the origin.  By continuity, for sufficiently small~$\alpha$ the points $P_0(\lambda',\alpha)$ and $P_1(\lambda',\alpha)$ remain in that ball.  Put $({p'},{q'})=(\lambda' p,\lambda' q)$.  For small~$\alpha$ the line $L({p'},{q'})$ intersects the cone~$C$ in points whose distance from the origin is $<\tfrac12|I|$, while the origin lies on the same side of the line.  Therefore the points $\gamma(\alpha)$, where $\gamma$ passes through $({p'},{q'})$, belong to $\Omega^\diamond$ for sufficiently small~$\alpha$.

Thus we have found $\lambda'>1$ such that the point $({p'},{q'})=(\lambda' p,\lambda' q)$ lies in $\mathrm{int}\,\Omega^\diamond$.  Lemma~\ref{lemma:inner_polar} then yields 
\[
\{\lambda(p,q)\mid\lambda>1\}\subset\mathrm{int}\,\Omega^\diamond.
\]
The theorem is proved.
\end{proof}

\medskip
To construct the functions $\cosh_\Omega$ and $\sinh_\Omega$ we need that for every point $\omega\in\partial\Omega$ there exists a point $\omega^\diamond\in\partial\Omega^\diamond$ such that $\omega\in L(\omega^\diamond)$, which is not always the case.  Indeed, it may happen that for some $\omega\in\Omega$ the line through $\omega$ separating $\Omega$ from the origin has the form $\{(x,y)\mid px-qy=0\}$ and therefore is not defined by any element of the antipolar set.

\begin{proposition}
\label{prop:good_points}
Let $\Omega\subset\mathbb{R}^2$ satisfy property \hyperref[main assumptions i * **]{$(i)$}.
\begin{enumerate}
\item If $\Omega$ possesses property \hyperref[main assumptions i * **]{$(*)$}, then for every point $\omega\in\partial\Omega$ there exists $\omega^\diamond\in\partial\Omega^\diamond$ such that $\omega\in L(\omega^\diamond)$.
\item If $\Omega$ possesses property \hyperref[main assumptions i * **]{$(**)$}, then for every point $\omega^\diamond\in\partial\Omega^\diamond$ there exists $\omega\in\partial\Omega$ such that $\omega\in L(\omega^\diamond)$.
\end{enumerate}
\end{proposition}

\begin{proof}
    We prove item 1.  
    A line that separates $\Omega$ from the origin and passes through the point $\omega\in\Omega$ exists by property \hyperref[main assumptions i * **]{$(i)$}.  
    Suppose that this line has the form $\{(x,y)\mid px-qy=0\}$, that is, it is not defined by any element of the antipolar set.  
    Then the ray $\{\lambda\omega\mid \lambda\ge 1\}$ is contained in $\partial\Omega$, which contradicts property \hyperref[main assumptions i * **]{$(*)$}.  
    Hence the parameters $(p,q)\in\Omega^\diamond$ cannot belong to the interior of the antipolar set, so $(p,q)\in\partial\Omega^\diamond$.

    We prove item 2.  
    Assume that the line $L(\omega^\diamond)$ does not intersect $\Omega$.  
    Then, by Lemma~\ref{lemma:inner_polar_gamma}, there exist arbitrarily small rotations of this line about the point nearest to the origin such that the rotated images of $L(\omega^\diamond)$ intersect $\Omega$.  
    The line $L(\lambda\omega^\diamond)$ with $\lambda>1$ has the same property by Lemma~\ref{lemma:small_angles}.  
    By Lemmas~\ref{lemma:inner_polar_criteria} and~\ref{lemma:inner_polar_gamma} we obtain $\lambda\omega^\diamond\in\partial\Omega^\diamond$, and by Lemma~\ref{lemma:inner_polar} we have $\{\lambda\omega^\diamond\mid\lambda\ge 1\}\subset\partial\Omega^\diamond$.  
    By Theorem~\ref{theorm:dual_properties} this implies that the set $\Omega$ cannot possess property \hyperref[main assumptions i * **]{$(**)$}, because in that case the set $\Omega^\diamond$ would fail to possess property \hyperref[main assumptions i * **]{$(*)$}.
\end{proof}

Thus, property \hyperref[main assumptions i * **]{$(*)$} for $\Omega$ guarantees that for every boundary point $\partial\Omega$ there exists a line $L(\omega^\diamond)$ with $\omega^\diamond\in\partial\Omega^\diamond$ passing through it;  
property \hyperref[main assumptions i * **]{$(**)$} for $\Omega$ guarantees that the boundary of $\Omega^\diamond$ contains no ``bad'' points, i.e.\ points whose corresponding lines do not intersect $\Omega$.
\subsection{Functions \texorpdfstring{\(\cosh _\Omega,\, \sinh _\Omega\)}{chΩ, shΩ}}

Assume that the set $\Omega$ satisfies Assumptions~\ref{main_assumption} and~\ref{dual_assmps}.  
Then, by Theorem~\ref{theorm:dual_properties}, its antipolar set $\Omega^\diamond$ also satisfies these assumptions, and by the biantipolar Theorem~\ref{theorem:biantipolar} we have $\Omega^{\diamond\diamond}=\Omega$.  
Thus the correspondence $\Omega\mapsto\Omega^\diamond$ is bijective on the class of sets that satisfy both assumptions.  
Moreover, every point of $\partial\Omega$ has a corresponding point of $\partial\Omega^\diamond$, and vice versa, in the sense of Proposition~\ref{prop:good_points}.  
Henceforth we shall always assume that $\Omega$ satisfies \emph{both} assumptions.

\medskip
\begin{defn}
Let $\Omega$ satisfy Assumptions~\ref{main_assumption} and~\ref{dual_assmps}, and fix a point $\omega_0=(x_0,y_0)\in\partial\Omega$.  
Let $O$ be the origin, $\omega=(x,y)\in\partial\Omega$, and let $\theta$ be the \emph{twice} the area of the contour\footnote{The contour is formed by the segments $O\omega_0$, $O\omega$ and the boundary arc of $\Omega$ between $\omega_0$ and $\omega$.} $O\omega_0\omega O$.  
We define
\begin{equation}
\label{ch_sh_defn}
    \cosh _\Omega\theta = x,\qquad
    \sinh _\Omega\theta = y.
\end{equation}
The functions \(\cosh _\Omega, \sinh _\Omega\) are defined either for all real ``angles''\footnote{Throughout we call the argument $\theta$ an ``angle'' when speaking of the functions $\cosh_\Omega,\sinh_\Omega$, meaning the area specified above.} $\theta$, or only on some interval or open ray, depending on whether the swept area may remain finite in the limit.
\end{defn}

\begin{example}
Let $\alpha>1$ and consider the domain $\Omega_\alpha$ from Example~\ref{exmpl:omega_alpha}.  
The domain of $\cosh_{\Omega_\alpha},\sinh_{\Omega_\alpha}$ is a finite interval if $\alpha>2$, and the whole real line if $\alpha\le 2$, as follows from convergence of the integral $\int_{1}^{+\infty}\! x-(x^\alpha-1)^{1/\alpha}\,dx$ for $\alpha>2$ and its divergence for $\alpha\le 2$.
\end{example}

The antipolar set $\Omega^\diamond$ likewise satisfies Assumptions~\ref{main_assumption} and~\ref{dual_assmps}, so we may construct the analogous functions for $\Omega^\diamond$.  
Take a fixed point $\omega_0^\diamond$ such that\footnote{Such a point exists by Proposition~\ref{prop:good_points}, but need not be unique.}  $\omega_0\in L(\omega_0^\diamond)$.  
With these fixed points $\omega_0,\omega_0^\diamond$ we obtain two pairs of functions satisfying the inequality
\begin{equation}\label{hyp_trig_ineq}
    \cosh _\Omega \theta\,\cosh _{\Omega^\diamond}\eta
    \;-\;
    \sinh _\Omega \theta\,\sinh _{\Omega^\diamond}\eta
    \;\ge\;1,
\end{equation}
where $\theta,\eta$ are arbitrary angles in the respective domains.  
Equality holds in~\eqref{hyp_trig_ineq} \emph{iff} the points
\((\cosh _\Omega\theta,\sinh _\Omega\theta)\) and \((\cosh _{\Omega^\diamond}\eta,\sinh _{\Omega^\diamond}\eta)\) correspond to each other in the sense of Proposition~\ref{prop:good_points}.

The correspondence of points yields, in general, a multi-valued correspondence of angles \(\theta\leftrightarrow\eta\).

\begin{defn}
For $\theta$ in the domain of \(\cosh_\Omega,\sinh_\Omega\) let $\theta^\diamond$ be the set of those $\eta$ in the domain of \(\cosh_{\Omega^\diamond},\sinh_{\Omega^\diamond}\) such that every point \((\cosh_\Omega\theta,\sinh_\Omega\theta)\) corresponds to \((\cosh_{\Omega^\diamond}\eta,\sinh_{\Omega^\diamond}\eta)\) (Proposition~\ref{prop:good_points}).  
Conversely, for $\eta$ in the domain of \(\cosh_{\Omega^\diamond},\sinh_{\Omega^\diamond}\) let ${}^\diamond\eta$ be the set of those $\theta$ in the domain of \(\cosh_\Omega,\sinh_\Omega\) corresponding to \((\cosh_{\Omega^\diamond}\eta,\sinh_{\Omega^\diamond}\eta)\).
\end{defn}

By Proposition~\ref{prop:good_points}, the sets $\theta^\diamond$ and ${}^\diamond\eta$ are non-empty for every admissible $\theta$ and $\eta$, respectively.  
In these terms we may state:

\begin{proposition}
\label{antipolar_corresp}
\[
    \cosh _\Omega \theta\,\cosh _{\Omega^\diamond}\eta
    -\sinh _\Omega \theta\,\sinh _{\Omega^\diamond}\eta
    \;=\;1
    \quad\Longleftrightarrow\quad
    \eta\in\theta^\diamond
    \;\Longleftrightarrow\;
    \theta\in{}^\diamond\eta.
\]
\end{proposition}

\medskip
We are now ready to formulate and prove a differentiation theorem for the functions
\(\cosh_\Omega\) and \(\sinh_\Omega\).

\begin{theorem}
\label{thrm: hyp_trig_derivatives}
The functions \(\cosh_\Omega\) and \(\sinh_\Omega\) are locally Lipschitz and therefore differentiable almost everywhere.  
If the set \(\theta^\diamond\) consists of the single element \(\eta\), then
\[
\frac{d}{d\theta}\,\cosh_\Omega(\theta)=\sinh_{\Omega^\diamond}\eta,
\qquad
\frac{d}{d\theta}\,\sinh_\Omega(\theta)=\cosh_{\Omega^\diamond}\eta.
\]

There exists at most a countable set of values of \(\theta\) where the left and right derivatives of \(\cosh_\Omega\) (respectively, of \(\sinh_\Omega\)) are different.  
At such a \(\theta\) the segments
\[
\bigl\{\sinh_{\Omega^\diamond}\eta\mid\eta\in\theta^\diamond\bigr\},
\qquad
\bigl\{\cosh_{\Omega^\diamond}\eta\mid\eta\in\theta^\diamond\bigr\}
\]
lie between the left and the right derivatives, for \(\cosh_\Omega\) and \(\sinh_\Omega\) respectively.
\end{theorem}

\begin{proof}
\emph{Step 1: local Lipschitz continuity.}
We parameterise the boundary \(\partial\Omega\) in classical polar form:
\((x,y)=R(\cos\varphi,\sin\varphi)\).
By the assumptions on~\(\Omega\), each boundary point corresponds to a unique angle
\(\varphi\in[0,2\pi)\); hence we may write
\[
(x,y)(\varphi)=R(\varphi)(\cos\varphi,\sin\varphi).
\]

The doubled sector area \(\theta=\theta(\varphi)\) associated with
\((x,y)(\varphi)\) is strictly monotone in~\(\varphi\) and locally Lipschitz,
and so is its inverse.  
Indeed, let \(\omega',\omega''\in\partial\Omega\) lie in a sufficiently small neighbourhood of a fixed \(\omega\in\partial\Omega\), and denote the corresponding angles by \(\varphi',\varphi''\) and the doubled areas by \(\theta',\theta''\).  
Then \(|\varphi'-\varphi''|\) is the angle between \(O\omega'\) and
\(O\omega''\) (here \(O\) is the origin), while \(|\theta'-\theta''|\) is the doubled signed area of the sector \(O\omega'\omega''O\).  
Locally, \(\partial\Omega\) is the graph of a convex function, and—except possibly at a countable set of points—there is a \emph{unique} supporting line not passing through the origin (Proposition~\ref{prop:good_points}).  
In this situation the area of the sector is bounded above by the area of the triangle on the same three vertices, and bounded below by the triangle determined by the intersections of \(O\omega'\) and \(O\omega''\) with that supporting line.  
Hence there exist constants \(C_1,C_2>0\) such that
\[
|\varphi'-\varphi''|\le C_1|\theta'-\theta''|,
\qquad
|\theta'-\theta''|\le C_2|\varphi'-\varphi''|.
\]

Because \(R(\varphi)\,\nu(\cos\varphi,\sin\varphi)\equiv1\) and the support function
\(\nu\) is concave, finite\footnote{Property \hyperref[main assumptions i * **]{$(*)$} implies that \(\nu\) vanishes only on \(\partial C\).} and bounded in a neighbourhood of \(\omega\), the function \(\nu\) is Lipschitz in a smaller neighbourhood.  Since \(\nu\) is separated from zero there, \(R(\varphi)\) is locally Lipschitz as well.  As \((\cosh_\Omega\theta,\sinh_\Omega\theta)=(x,y)(\varphi(\theta))\) and \(\varphi(\theta)\) is locally Lipschitz, both \(\cosh_\Omega\) and \(\sinh_\Omega\) are locally Lipschitz.

\medskip\noindent
\emph{Step 2: differentiation at regular points.}
A Lipschitz function is differentiable almost everywhere.  
At every boundary point \(\omega\in\partial\Omega\) with a \emph{unique} supporting line
\(L=\{(x,y)\mid px-qy=1\}\) (necessarily with \((p,q)\in\partial\Omega^\diamond\)),
set
\[
L=\{l_t=\omega+t(q,p)\mid t\in\mathbb{R}\}.
\]
Because \(L\) is the only support line, we may write
\(l_t=\omega_t+\alpha_t\) where \(\omega_t=\partial\Omega\cap\{\lambda l_t\mid\lambda\in\mathbb{R}\}\)
and \(|\alpha_t|=o(t)\) as \(t\to0\).

Let \(\theta_t\) be the doubled signed area of the contour \(O\omega\omega_tO\)
and \(\theta=\theta_0\).  Then
\[
\theta_t=\theta+2|[Ol_t,O\omega]|+\beta_t,
\]
where \(\beta_t\) is the area of the convex contour \(\omega\omega_tl_t\);
hence \(|\beta_t|=O(t^2)\).

Write \((p,q)=(\cosh_{\Omega^\diamond}\eta,\sinh_{\Omega^\diamond}\eta)\) with
\(\{\eta\}=\theta^\diamond\).  A direct computation gives
\begin{align*}
\frac{\partial(\cosh_\Omega,\sinh_\Omega)}{\partial\theta}
  &=\lim_{t\to0}
    \frac{(\cosh_\Omega\theta_t,\sinh_\Omega\theta_t)-(\cosh_\Omega\theta,\sinh_\Omega\theta)}
         {\theta_t-\theta}\\[2mm]
  &=\lim_{t\to0}
    \frac{(\sinh_{\Omega^\diamond}\eta,\cosh_{\Omega^\diamond}\eta)t-\alpha_t}
         {2|[Ol_t,O\omega]|+\beta_t}\\[2mm]
  &=\lim_{t\to0}
    \frac{(\sinh_{\Omega^\diamond}\eta,\cosh_{\Omega^\diamond}\eta)t+o(t)}
         {t+O(t^{2})}
  =(\sinh_{\Omega^\diamond}\eta,\cosh_{\Omega^\diamond}\eta).
\end{align*}

\medskip\noindent
\emph{Step 3: points with several supports.}
If \(\omega\) admits more than one supporting line, these lines correspond to the elements
\(\{(\cosh_{\Omega^\diamond}\eta,\sinh_{\Omega^\diamond}\eta)\mid\eta\in\theta^\diamond\}\), which form a segment; its endpoints give the right- and left-hand tangents to \(\Omega\) at \(\omega\).  Repeating the above argument with one-sided limits \(t\to0^{+}\) and \(t\to0^{-}\) yields the remaining part of the theorem.
\end{proof}

Any point \((x,y)\in\mathbb{R}_{+}\Omega\) can be written in the form  
\[
(x,y)=R\bigl(\cosh_\Omega\theta,\sinh_\Omega\theta\bigr),\qquad R\ge 0.
\]

\begin{proposition}
Let the curve \((x,y)(t)\in\mathbb{R}_{+}\Omega\) be absolutely continuous and never pass through the origin.  
Then the functions \(R(t)\) and \(\theta(t)\) defined by
\((x,y)=R(\cosh_\Omega\theta,\sinh_\Omega\theta)\)
are also absolutely continuous.

If, at the value \(\theta=\theta(t)\), the functions \(\cosh_\Omega,\sinh_\Omega\) are differentiable\footnote{Equivalently, the set \(\theta^\diamond\) consists of a single element.}, then
\[
\dot R
   =\dot x\,\cosh_{\Omega^\diamond}\eta
     -\dot y\,\sinh_{\Omega^\diamond}\eta,
\qquad
\dot\theta
   =\frac{\dot y\,\cosh_\Omega\theta-\dot x\,\sinh_\Omega\theta}{R},
\qquad
\{\eta\}=\theta^\diamond.
\]
\end{proposition}

\begin{proof}
Fix a moment \(t_0\).  
Because the support function \(\nu\) is concave, finite\footnote{Property \hyperref[main assumptions i * **]{$(*)$} implies that \(\nu\) vanishes only on \(\partial C\).} and bounded near the point \((x,y)(t_0)\), it is Lipschitz there; consequently the function \(R(t)=\nu\bigl(x(t),y(t)\bigr)\) is absolutely continuous in a neighbourhood of \(t_0\).

The function
\[
\frac{x(t)\dot y(t)-\dot x(t)y(t)}{R^{2}(t)}
\]
is integrable near \(t_0\) because \(R(t)\) is bounded away from zero.  
Applying Green’s formula to the curve
\(\bigl(\tfrac{x}{R},\tfrac{y}{R}\bigr)(t)\in\partial\Omega\) we obtain
\[
\theta(t)=\int_{t_0}^{t}
        \frac{x(s)\dot y(s)-\dot x(s)y(s)}{R^{2}(s)}\,ds+\theta(t_0),
\]
so \(\theta(t)\) is absolutely continuous as well.

Derivatives \(R'(t)\) and \(\theta'(t)\) exist for almost every \(t\).  
Assume they exist at the chosen \(t\) and that \(\theta(t)^\diamond=\{\eta\}\), i.e.\ \(\cosh_\Omega,\sinh_\Omega\) are differentiable at \(\theta(t)\).  Then
\[
\dot x
   =\frac{d}{dt}\!\bigl(R\cosh_\Omega\theta\bigr)
   =\dot R\,\cosh_\Omega\theta
     +R\,\sinh_{\Omega^\diamond}\eta\,\dot\theta,
\qquad
\dot y
   =\frac{d}{dt}\!\bigl(R\sinh_\Omega\theta\bigr)
   =\dot R\,\sinh_\Omega\theta
     +R\,\cosh_{\Omega^\diamond}\eta\,\dot\theta.
\]

Multiply the first equality by \(\cosh_{\Omega^\diamond}\eta\), the second by \(-\sinh_{\Omega^\diamond}\eta\), and add.  Using Proposition~\ref{antipolar_corresp},
\[
\cosh_\Omega\theta\,\cosh_{\Omega^\diamond}\eta
   -\sinh_\Omega\theta\,\sinh_{\Omega^\diamond}\eta
   =1,
\]
we obtain the formula for \(\dot R\); substituting it back gives the expression for \(\dot\theta\).
\end{proof}

\section{Lightlike and timelike extremals}
\label{sec: light_time_like}

In this section we prove a theorem that splits extremals in a sub-Lorentzian problem into \emph{timelike} and \emph{lightlike} classes.

\medskip
We start with two auxiliary propositions that will be used together with the Pontryagin Maximum Principle (PMP) when searching for extremals of sub-Lorentzian problems.

\begin{proposition}
    \label{pmp_antinorms}
    Let $h=(h_1,h_2)\in\mathbb{R}^{2*}$, $h\neq0$, let $\nu$ be an antinorm on the plane (Definition~\ref{defn: antinorm}), and let $\Omega$ be its unit ball (Definition~\ref{defn: antinorm_ball}),
    \[
      \Omega=\{u\in\mathbb{R}^2\mid \nu(u)\ge1\}.
    \]
    Define for $u\in\dom\nu=C=\mathrm{cl}\,\mathbb{R}_+\Omega$
    \[
      \mathcal{H}(h,u)=\langle h,u\rangle+\nu(u).
    \]
    Then $\sup_{u\in C}\mathcal{H}(h,u)$ is finite, and
    \[
      \sup_{u\in C}\mathcal{H}(h,u)=0
      \Longleftrightarrow
      (-h_1,h_2)\in\Omega^\diamond .
    \]
    In this case $\mathcal{H}(h,u)\le0$ for all $u\in C$ and $\mathcal{H}(h,0)=0$.
\end{proposition}

\begin{proof}
The supremum is either $0$ or $+\infty$.  
If some $u\in C$ satisfies $\mathcal{H}(h,u)>0$, then
$\lambda u\in C$ for every $\lambda\ge0$ and
$\mathcal{H}(h,\lambda u)=\lambda\mathcal{H}(h,u)\to+\infty$,
so the supremum is infinite.
Conversely, $\mathcal{H}(h,0)=0$, hence the supremum cannot be negative.

Because $\mathcal{H}(h,\cdot)$ is continuous on~$C$,
\[
  \sup_{u\in C}\mathcal{H}(h,u)=0
  \Longleftrightarrow
  \langle h,u\rangle+\nu(u)\le0\quad\forall u\in\Omega .
\]
Since $\nu(u)\ge1$ on $\Omega$, the last inequality is equivalent to
$\langle-h,u\rangle\ge1$ on $\Omega$, i.e.\ $(-h_1,h_2)\in\Omega^\diamond$.
The reverse implication is obtained by scaling $u$ to unit antinorm.
\end{proof}

\begin{proposition}
\label{H_argmax}
Under the assumptions of Proposition~\ref{pmp_antinorms}, suppose in addition that
$\Omega$ satisfies Assumption~\ref{dual_assmps}. Then
\begin{enumerate}
\item If $h\in\mathrm{int}\,\Omega^\diamond$, the maximum of $\mathcal{H}(h,u)$ over $u\in C$ is attained only at $u=0$.
\item If $h\in\partial\Omega^\diamond$, write $(-h_1,h_2)=(\coshi_{\Omega^\diamond}\eta,\sinhi_{\Omega^\diamond}\eta)$ for some ``angle'' $\eta\in\mathbb{R}$.  
      Then $\max_{u\in C}\mathcal{H}(h,u)$ is attained at every point of
      \[
        \{\lambda(\coshi_\Omega\theta,\sinhi_\Omega\theta)\mid\lambda\ge0,\ \theta\in{}^\diamond\eta\}.
      \]
\end{enumerate}
\end{proposition}

\begin{proof}
If $(-h_1,h_2)\in\Omega^\diamond$ and the maximum were reached at $u\in\partial C$, then
$\nu(u)=0$ by property \hyperref[main assumptions i * **]{$(*)$} and hence
$\mathcal{H}(h,u)=0$, contradicting $h\notin\Omega^\diamond$; thus $u\in\mathbb{R}_{+}\Omega$.
Writing $u=R(\coshi_\Omega\theta,\sinhi_\Omega\theta)$,  
$h=K(-\coshi_{\Omega^\diamond}\eta,\sinhi_{\Omega^\diamond}\eta)$ with $K\ge1$,
inequality~\eqref{hyp_trig_ineq} gives
\[
\mathcal{H}(h,u)\le R(1-K),
\]
with equality exactly when $\theta\in{}^\diamond\eta$; the bound is $0$ only if $K=1$.
\end{proof}

\medskip
\subsection*{Setting of the control problem}

Let $M$ be a smooth manifold of arbitrary finite dimension, and let
$f_1,f_2$ be smooth vector fields that are linearly independent at every point.
Put $\Delta=\langle f_1,f_2\rangle$ and fix an antinorm~$\nu$ in~$\mathbb{R}^2$.
Consider the problem of maximising the \emph{length}
\[
  \int_0^{T}\!\nu(u_1,u_2)\,dt\;\longrightarrow\;\max
\]
over Lipschitz curves satisfying
\[
  q(0)=q_0,\quad q(T)=q_1,\quad \dot q=u_1f_1+u_2f_2,
  \qquad T>0,\ (u_1,u_2)\in\mathbb{R}^2 .
\]

\smallskip
Applying Pontryagin’s Maximum Principle we write
\[
  \dot p=-\mathcal{H}'_q,\qquad
  \dot q=\mathcal{H}'_p,\qquad
  \mathcal{H}(q,p,u)\to\max_{u\in\mathbb{R}^2},
\]
where
\[
  \mathcal{H}(q,p,u)=\langle p,u_1f_1+u_2f_2\rangle-\lambda_0\nu(u_1,u_2),
  \qquad \lambda_0\in\{0,-1\}.
\]

Introduce
\[
  h_k=\langle p,f_k\rangle,\ k=1,2,\quad
  h_3=\langle p,[f_1,f_2]\rangle=\{h_1,h_2\},
\]
so that
\[
  \mathcal{H}(q,p,u)=\langle h,u\rangle-\lambda_0\nu(u),
  \qquad h=(h_1,h_2),\ u=(u_1,u_2).
\]

Note that $\dot q(t)=0$ iff $u(t)=0$; by re-parameterising the curve and adjusting~$T$ we can assume $u(t)\neq0$ for a.e.\ $t$.

\begin{theorem}
\label{thrm: light_space_like}
Let the unit ball $\Omega=\{u\in\mathbb{R}^2\mid\nu(u)\ge1\}$ satisfy Assumption~\ref{dual_assmps}.  
Let $(q(t),p(t),u(t))$ be a PMP extremal with $u(t)\neq0$ a.e.
\begin{enumerate}
\item
  If $\lambda_0=-1$ then $\displaystyle\int_0^{T}\!\nu(u)\,dt>0$.  
  In this \emph{timelike} case $u(t)\in\mathbb{R}_{+}\Omega$ a.e., so $\nu(u)=1$ a.e.\ under natural parametrisation.  
  Then $(-h_1,h_2)\in\Omega^\diamond$ and
  \[
     (h_1,h_2)=\bigl(-\coshi[\eta],\sinhi[\eta]\bigr)
     \quad\text{for some }\eta\in\mathbb{R},
     \qquad
     u=\bigl(\coshi_\Omega\theta,\sinhi_\Omega\theta\bigr),
     \quad\theta\in{}^\diamond\eta .
  \]
\item
  If $\lambda_0=0$ (the \emph{lightlike} case) and $(h_1,h_2)(t)\neq0$ at some time~$t$, then
  \[
    (h_1,h_2)(t)\in\partial C^{*},\qquad
    u(t)\in\partial C,\qquad
    \langle h(t),u(t)\rangle=0,
  \]
  where $C=\dom\nu=\mathrm{cl}\,\mathbb{R}_{+}\Omega$ and
  $C^{*}=\{h\in\mathbb{R}^{2*}\mid\langle h,u\rangle\le0\;\forall u\in C\}$.
  If $h_3(t)$ changes sign only finitely many times, say $k$ times, then the control can switch from one open ray of $\partial C^{*}$ to the other at most $k+1$ times.
\end{enumerate}
\end{theorem}

\begin{proof}
Because $\Omega$ is the unit ball of an antinorm, it satisfies Assumption~\ref{main_assumption};
Assumption~\ref{dual_assmps} holds by hypothesis.

\smallskip\noindent
\emph{(1)} Set $\lambda_0=-1$.  
If $h=(h_1,h_2)=0$ at some instant, the supremum of $\mathcal{H}$ is $+\infty$.
Otherwise, Proposition~\ref{H_argmax} applies and yields the stated description.

\smallskip\noindent
\emph{(2)} For any $\lambda_0\in\{0,-1\}$ the maximised value of $\mathcal{H}$
over $u\in C$ is always~$0$: if some $u$ gave $\mathcal{H}>0$, the supremum would be infinite; meanwhile $\mathcal{H}(h,0)=0$.  
With $\lambda_0=0$ and $h\neq0$,  
\[
  \max_{u\in C}\mathcal{H}(h,u)=\langle h,u\rangle
  \quad\text{occurs iff } h\in C^{*},
\]
otherwise a $u\in C$ would give $\mathcal{H}>0$.

If $h\in\mathrm{int}\,C^{*}$, the maximum is attained only at $u=0$, contradicting $u(t)\neq0$ a.e.; thus such $h$ cannot occur after any time~$\tau$.

If $h=0$, any $u\in C$ maximises $\mathcal{H}$.
If $h\in\partial C^{*}$, the maximum is achieved at every $u$ on the ray of
$\partial C$ satisfying $\langle h,u\rangle=0$; hence
$u=\lambda(h_2,-h_1)$ with $\lambda$ of the sign determined by the chosen ray.

Assume $h\neq0$ at some $\tau$.  
By PMP we have
\[
  \dot h_1=-u_2\,h_3,\quad
  \dot h_2= u_1\,h_3,
  \quad h_3=\langle p,[f_1,f_2]\rangle .
\]
Since $p$ is Lipschitz and $f_k$ are smooth, $h_3$ is Lipschitz, hence continuous.  
Near~$\tau$,
$\dot h_1=\lambda(t)h_1h_3,\ \dot h_2=\lambda(t)h_2h_3$ for some $\lambda(t)\neq0$,
so $h(t)$ stays on the same ray of $\partial C^{*}$ until possibly crossing the origin when $h(t)=0$.  
While $h_3(t)$ keeps a fixed sign, $h(t)$ can cross the origin—and hence switch
to the opposite ray—at most once.  
Therefore, if $h_3$ changes sign $k$ times, at most $k+1$ such switches can occur.
\end{proof}

\section{The Lorentzian problem on the Lobachevsky plane}
\label{sec:lobachevski}

In this paragraph we give a complete description of extremals for the Lorentzian
problems of Section~\ref{subl: affr} on the Lobachevsky plane
$\mathrm{Aff}_+\mathbb{R}$.  Because on the Lobachevsky plane the set
$U=\{(u_1,u_2)\in\mathbb{R}^2\mid\nu_\UU(u)\ge1\}$ is two-dimensional, we shall
write $\Omega=U$ throughout this paragraph, so that the notation matches the
functions $\cosh_\Omega$ and $\sinh_\Omega$; for convenience we also set
$\nu=\nu_\Omega$.  Everywhere in this paragraph we assume that $\Omega$
satisfies Assumptions~\ref{main_assumption} and~\ref{dual_assmps}.

Recall that $\mathbb{R}_{+}=\{\lambda\in\mathbb{R}\mid\lambda\ge0\}$ and
$C=\mathrm{cl}\,\mathbb{R}_{+}\Omega$.

\medskip\noindent
We first write, in coordinates, the problem of finding length-maximising curves
for the Lorentzian metric on the Lobachevsky plane.  The velocity vector of an
admissible curve $q(t)=(a(t),b(t))\in\mathrm{Aff}_+\mathbb{R}
  =\{(a,b)\mid a,b\in\mathbb{R},\ b>0\}$
lies almost everywhere in the image of $C$ under the differential of the left
translation:
\[
  \dot q(t)=dL_{q(t)}u(t),\qquad u(t)\in C .
\]

The left translation on the group $\mathrm{Aff}_+(\mathbb{R})$ is
\[
  (L_{(a,b)}(c,d))x
     =(a,b)\bigl((c,d)x\bigr)
     =a+b(c+dx)=(a+bc,bd)x,
\]
so its differential is
\[
  dL_{(a,b)}=\begin{pmatrix} b&0\\[2pt] 0&b \end{pmatrix}.
\]

Hence we arrive at the optimal-control problem\footnote{The initial point can
be taken to be the unit element $(0,1)$ of the group.}
\begin{equation}
\label{eq:Lobachevski_problem_in_coordinates}
\begin{gathered}
  \int_0^T\!\nu(u)\,dt\ \longrightarrow\ \max,\\[4pt]
  a(0)=0,\; b(0)=1,\qquad a(T)=a_1,\; b(T)=b_1,\\[4pt]
  \begin{cases}
    \dot a=bu_1,\\
    \dot b=bu_2,
  \end{cases}\qquad
  u=(u_1,u_2)\in C .
\end{gathered}
\end{equation}
The corresponding vector fields are $f_1=b\,\partial/\partial a$,
$f_2=b\,\partial/\partial b$.

According to Pontryagin’s Maximum Principle, each length-maximising curve is
the projection to $\mathrm{Aff}_+\mathbb{R}$ of an extremal in the cotangent
bundle $T^*\mathrm{Aff}_+\mathbb{R}$.

The PMP Hamiltonian is
\[
  \mathcal{H}(q,p,u)
    =\langle p,bu\rangle-\lambda_0\nu(u)
    =\langle h,u\rangle-\lambda_0\nu(u),
\]
where $\lambda_0\in\{0,-1\}$ is constant along every extremal,
$h_k=\langle p,f_k\rangle=bp_k$ for $k=1,2$, and $h=(h_1,h_2)$.

Put $h_3=\{h_1,h_2\}=\langle p,[f_1,f_2]\rangle$.
Since $[f_1,f_2]=-f_1$, we have $h_3=-h_1$.

\begin{proposition}
\label{prop:time_or_light_like_affr}
For the problem under consideration the extremals split into two types:
\begin{enumerate}
\item \emph{normal} extremals, which are timelike (i.e.\ $u\in\mathbb{R}_{+}\Omega$
      almost everywhere);
\item \emph{abnormal} extremals, which are lightlike
      (the antinorm is zero almost everywhere).
\end{enumerate}
\end{proposition}
\begin{proof}
Apply Theorem~\ref{thrm: light_space_like} and observe that $h(t)\neq0$ for
a.e.\ $t$ when $\lambda_0=0$.  
Indeed, $h(t)=0$ at some moment iff $p(t)=0$ (because $b>0$), while
Pontryagin’s Maximum Principle guarantees $p(t)\neq0$ for all~$t$.
\end{proof}

By separating trajectories into the two types above, we obtain the following
result.

\begin{theorem}
\label{thm:affine_lorentz_extremals}
Assume that the set $\Omega$ satisfies Assumptions~\ref{main_assumption} and
\ref{dual_assmps}.  
Then, in problem~\eqref{eq:Lobachevski_problem_in_coordinates}, every extremal
is either \emph{abnormal lightlike} or \emph{normal timelike}.

\begin{enumerate}
\item[\textup{(i)}]%
  Let $(a(t),b(t))$ be the projection to $\mathrm{Aff}_+\mathbb{R}$ of a
  timelike extremal.  Then either  
  \begin{itemize}
  \item[(i.a)] the extremal is not singular with respect to any edge of
        $\Omega$ on any set of positive measure; or
  \item[(i.b)] the control $(u_1,u_2)$ lies, for a.e.\ $t$, on a
        (maximal) facet $F\subset\partial\Omega$ whose supporting line
        is horizontal\footnote{%
        The boundary of $U=\Omega$ may contain no points with a vertical
        supporting line (in which case (i.b) never occurs), or it may contain
        \emph{at most one} such facet $F$ by the ray property.  In the latter
        case $F$ can be a point or a segment.}
        (if $\dim F>0$, the extremal is singular along $F$ for every
        $t\in[0,T]$).
  \end{itemize}

\item[\textup{(ii)}]%
  If $(a(t),b(t))$ is the projection of a lightlike extremal, then it is
  abnormal and the control belongs to one of the two boundary rays\footnote{%
  Recall that $\partial C=l_0\cup l_1$.} $l_0$ or $l_1$ for a.e.\ $t$.
\end{enumerate}

Moreover:

\begin{itemize}
\item
  In case \textup{(i.a)} there exist constants $c_1\neq0$ and $c_2$ such that
  \[
      a=c_1\,\sinh_{\Omega^\diamond}\eta + c_2,\qquad
      b=c_1\,\cosh_{\Omega^\diamond}\eta,
  \]
  where $\cosh_{\Omega^\diamond}\eta(t)\neq0$ for all $t$, and the “angle’’
  $\eta(t)$ satisfies the quadrature
  \[
      t=\int \frac{d\eta}{\cosh_{\Omega^\diamond}\eta}.
  \]

\item
  In case \textup{(i.b)} we have
  $b(t)=c_3e^{c_4t}>0$ with constants $c_3,c_4$, valid for every~$t$; the
  constant $c_4$ equals the second coordinate of the horizontal facet~$F$.

\item
  In case \textup{(ii)} there are exactly two trajectories (whose Lorentzian
  length is zero) issuing from the identity and parallel to the rays
  $l_{0,1}$.
\end{itemize}
\end{theorem}

Thus, timelike extremals of type~\textup{(i.a)} are described as follows: one
reflects the antipolar set $\Omega^\diamond$ across the diagonal
$(p_1,p_2)\mapsto(p_2,p_1)$, scales it by a factor $c_1>0$, and shifts it
horizontally by an arbitrary amount~$c_2$.  
Any arc of the resulting curve lying in the upper half-plane $b>0$ is the
projection of some non-singular extremal.  
Case~\textup{(i.b)} resembles vertical geodesics in the Poincaré half-plane
model of the Lobachevsky plane.  
Case~\textup{(ii)} consists of the two lightlike extremals whose Lorentzian
length is zero.

\begin{proof}
We first rewrite the Pontryagin system.  In the variables
$h=(h_1,h_2)=b(p_1,p_2)$ the PMP equations are
\begin{equation}
\label{eq:Hamiltonian_system_Lobachevski_plane}
\begin{cases}
\dot a   = b\,u_1,\\ 
\dot b   = b\,u_2,\\[2pt]
\dot h_1 =-u_2\,h_3 = u_2\,h_1,\\
\dot h_2 = u_1\,h_3 =-u_1\,h_1,\\[2pt]
\displaystyle
\mathcal{H}(h,u)=\langle h,u\rangle-\lambda_0\nu(u)\to\max_{u\in C},
\end{cases}
\end{equation}
with $h_3=\{h_1,h_2\}=-h_1$.

\medskip\noindent
\emph{Lightlike case.}
By Proposition~\ref{prop:time_or_light_like_affr}, lightlike extremals correspond to $\lambda_0=0$.  Then $u(t)\in\partial C$ for a.e.\ $t$
(Theorem~\ref{thrm: light_space_like}).  
Switches between the two rays of $\partial C$ occur only when $h_3(t)=-h_1(t)$ changes sign.  But $h_3$ satisfies
\(
\dot h_3=-u_2h_3=-u_2h_3,
\)
so $h_3$ is either identically~$0$ or sign-constant; hence at most one switch is possible.  
If $h_3\equiv0$ then $h_1\equiv0$ and $\dot h_2=u_1h_3\equiv0$, while
$h_2\ne0$ by Proposition~\ref{prop:time_or_light_like_affr}.  Thus $(h_1,h_2)\neq0$ for all $t$ and the control stays on a fixed boundary ray $l_k$.  
Projecting to $\mathrm{Aff}_+\mathbb{R}$ yields the two lightlike trajectories of case~(iii).

\medskip\noindent
\emph{Timelike case.}
With $\lambda_0=-1$, the maximised Hamiltonian is
\begin{equation}
\label{eq:hamiltonian_lobachevski_plane}
   \mathcal{H}(h,u)=\langle h,u\rangle+\nu(u)\to\max_{u\in C}.
\end{equation}
By Theorem~\ref{thrm: light_space_like} the maximum equals $0$ and
$(-h_1,h_2)\in\Omega^\diamond$.  
If $(-h_1,h_2)\in\partial\Omega^\diamond$ we write
\begin{equation}
\label{eq:Lobachevski_plane_trigonometry_change}
   h_1=-\cosh_{\Omega^\diamond}\eta,\qquad
   h_2=\sinh_{\Omega^\diamond}\eta .
\end{equation}
Normalising to $\nu(u)=1$, the maximiser is
$u=(\coshi_\Omega\theta,\sinhi_\Omega\theta)$ with $\theta\in{}^\diamond\eta$.

Applying Theorem~\ref{thrm: hyp_trig_derivatives} and inserting
\eqref{eq:Lobachevski_plane_trigonometry_change} into
\eqref{eq:Hamiltonian_system_Lobachevski_plane} we obtain
\[
  \dot\eta\,\sinh_\Omega\theta = \sinh_\Omega\theta\,\cosh_{\Omega^\diamond}\eta,
  \qquad
  \dot\eta\,\cosh_\Omega\theta = \cosh_\Omega\theta\,\cosh_{\Omega^\diamond}\eta ,
\]
whence $\dot\eta=\cosh_{\Omega^\diamond}\eta$ and
$d\eta=\cosh_{\Omega^\diamond}\eta\,dt$.

\smallskip
If $\cosh_{\Omega^\diamond}\eta(t)\not\equiv0$, the extremal is nowhere singular.
From $\dot b=b\sinh_\Omega\theta$ and the relations above we deduce
\[
  b=c_1\,\cosh_{\Omega^\diamond}\eta,\qquad
  a=c_1\,\sinh_{\Omega^\diamond}\eta+c_2,
\]
with constants $c_1\neq0$ and $c_2$, while
\(
  t=\int d\eta/\cosh_{\Omega^\diamond}\eta.
\)
This is case~(i).

\smallskip
If instead $\cosh_{\Omega^\diamond}\eta(t)\equiv0$, then
$\eta(t)\equiv\eta_0$ for the unique angle $\eta_0$ with
$\cosh_{\Omega^\diamond}\eta_0=0$, so
$\sinh_{\Omega^\diamond}\eta_0\neq0$ is constant.  
Consequently $\sinh_\Omega\theta$ is constant and
$u(t)$ lies on a horizontal facet $F$ of $\partial\Omega$; writing
$\sinh_\Omega\theta=\mathrm{const}=c_4$, integration of
$\dot b=b\,c_4$ gives $b(t)=c_3e^{c_4t}$,
which is case~(ii).
\end{proof}

\section{Sub-Lorentzian problems on three-dimensional unimodular Lie groups}
\label{sec:3d_groups}

In this section we explicitly integrate the \emph{vertical subsystem} of the
Pontryagin Maximum Principle for sub-Lorentzian problems on the
three-dimensional unimodular Lie groups
\[
    G = SL(2),\ SU(2),\ SE(2),\ SH(2),\ \mathbb{H}_3.
\]
The general formulation is still~\eqref{eq:subl_general_problem}, where
$U\subset T_1G$ is a convex two-dimensional set.  
Because $U$ is two-dimensional, we henceforth write $\Omega=U$ so that the
notation matches the convex trigonometry, and we put
$\nu=\nu_\Omega$ for brevity.  
Assume that $\Omega$ satisfies Assumptions~\ref{main_assumption}
and~\ref{dual_assmps}.  
By the Pontryagin principle the Hamiltonian is
\[
   \mathcal{H}
      =\langle p,dL_q u\rangle-\lambda_0\nu(u)
      \ \longrightarrow\ \max_{u\in C},
   \qquad
   \lambda_0\in\{0,-1\},
\]
with $p\in T^*_qG$.

Introduce left-invariant coordinates on each fibre of $T^*G$ by setting
\[
   h_k=\langle p,dL_qf_k\rangle,\qquad k=1,2,3,
\]
where $f_k$ is a Lie-algebra basis
(see Section~\ref{ssec:3d_lie_groups_porblem_statement}).  
On $T^*_1G$ (i.e.\ at $q=1$) the Lie–Poisson brackets of the Hamiltonians
$h_k$ read
\[
   \{h_j,h_k\}
     =\{\langle p,f_j\rangle,\langle p,f_k\rangle\}
     =\langle p,[f_j,f_k]\rangle .
\]
Using the structure constants in~\eqref{eq:3d_lie_brackets} we have
\[
   \{h_1,h_2\}=h_3,\qquad
   \{h_3,h_1\}=a\,h_2,\qquad
   \{h_3,h_2\}=b\,h_1,
\]
where $a,b\in\{0,\pm1\}$ are constants.  
In the coordinates $h_j$ the Pontryagin Hamiltonian becomes
\[
   \mathcal{H}
      =h_1u_1+h_2u_2-\lambda_0\nu(u)
      \ \longrightarrow\ \max_{u\in C}.
\]

Hence the PMP Hamiltonian equations are
$\dot h_j=\{\mathcal{H},h_j\}$ and $\dot q=\{\mathcal{H},q\}$.  
This splits off a three-dimensional subsystem, independent of $q$,
\begin{equation}
\label{eq:vertical_subsystem}
\begin{gathered}
   h_1u_1+h_2u_2-\lambda_0\nu(u)\to\max_{u\in C},\\[2pt]
   \begin{cases}
     \dot h_1=-h_3u_2,\\
     \dot h_2= h_3u_1,\\
     \dot h_3=-a\,h_2u_1- b\,h_1u_2.
   \end{cases}
\end{gathered}
\end{equation}
Following standard terminology, we call~\eqref{eq:vertical_subsystem} the
\emph{vertical part} of the PMP ODE system for left-invariant problems: it has
half the dimension (three instead of six), and it can be integrated
independently of the horizontal equations
$\dot q=\{\mathcal{H},q\}$ on the group~$G$.  
Once~\eqref{eq:vertical_subsystem} is integrated, its solution yields the
optimal control $u(t)$—which therefore does not depend on the solution of the
lower-level equation $\dot q=\{\mathcal{H},q\}$.
\begin{proposition}
\label{prop: not_abnormal_11}
Every extremal of \eqref{eq:vertical_subsystem} can be re-parametrised so that $u\neq0$ for almost every $t\in[0,T]$. In such a parametrisation we have $(h_1,h_2)\neq0$ for almost every $t\in[0,T]$.  
Consequently, extremals split into \emph{lightlike} and \emph{timelike} classes.
\end{proposition}

\begin{proof}
If $u=0$ on some time interval, the extremal is constant there; by shortening
$T$ and re-parametrising we may assume $u\neq0$ a.e.\ on $[0,T]$.  
Suppose, on the contrary, that $h_1(t)=h_2(t)=0$ on a set
$\mathbb{T}\subset[0,T]$ of positive measure; almost every point of
$\mathbb{T}$ is a Lebesgue (hence limit) point.
Since $h_1,h_2$ are Lipschitz, they are differentiable a.e.; therefore, for
a.e.\ $t\in\mathbb{T}$,
\[
  \dot h_1(t)=-h_3(t)u_2(t)=0,\qquad
  \dot h_2(t)= h_3(t)u_1(t)=0.
\]
Because $u\neq0$, we conclude $h_3(t)=0$ for a.e.\ $t\in\mathbb{T}$, which is
impossible: if $h_1(t_0)=h_2(t_0)=h_3(t_0)=0$ at some $t_0$, then
$p(t_0)=0$, but $p(t)$ obeys a linear ODE and would be identically zero,
forcing $\mathcal{H}\equiv0$, contradicting PMP.  

Since $\Omega$ satisfies Assumptions~\ref{main_assumption}
and~\ref{dual_assmps}, Theorem~\ref{thrm: light_space_like} yields the desired
lightlike/timelike dichotomy.
\end{proof}

\begin{theorem}
Extremals in left-invariant sub-Lorentzian problems on \emph{all}
three-dimensional unimodular Lie groups are either  
\begin{itemize}
\item normal and timelike, or  
\item abnormal and lightlike.
\end{itemize}

On every extremal of \eqref{eq:vertical_subsystem} one has
$\displaystyle\max_{u\in\Omega}\mathcal{H}=0$.  
Along a lightlike extremal $u(t)\in\partial C$ for almost every
$t\in[0,T]$.  

For every timelike extremal there exists a constant $E\in\mathbb{R}$ such that
\[
\begin{cases}
h_1=-\cosh_{\Omega^\diamond}\eta,\\
h_2=\sinh_{\Omega^\diamond}\eta,\\
h_3=\pm\sqrt{\,E-a\bigl(\sinh_{\Omega^\diamond}\eta\bigr)^2
              +b\bigl(\cosh_{\Omega^\diamond}\eta\bigr)^2},
\end{cases}
\qquad
\begin{cases}
u_1=\cosh_\Omega\theta,\\
u_2=\sinh_\Omega\theta,
\end{cases}
\]
with
\begin{equation}
\label{eq:energy_inequality}
   E\;\ge\;a\bigl(\sinh_{\Omega^\diamond}\eta\bigr)^2
          -b\bigl(\cosh_{\Omega^\diamond}\eta\bigr)^2,
\end{equation}
and the angles correspond in the sense $\theta\in{}^\diamond\eta$.

If on an interval $(\alpha,\beta)$ the inequality
\eqref{eq:energy_inequality} is strict, then on that interval $\eta$ satisfies
the quadrature\footnote{The sign $\pm$ may change from one connected sub-interval of $(\alpha,\beta)$ to another, but is fixed on each such sub-interval.}
\[
   t=\pm\int
        \frac{d\eta}{
               \sqrt{\,E-a\bigl(\sinh_{\Omega^\diamond}\eta\bigr)^2
                        +b\bigl(\cosh_{\Omega^\diamond}\eta\bigr)^2}}.
\]

If an extremal is singular along an edge $F$ on a set
$\mathbb{T}\subset[0,T]$ of positive measure, then $\eta\equiv\mathrm{const}$
on $\mathbb{T}$ and \eqref{eq:energy_inequality} holds with equality.
\end{theorem}

\begin{proof}
By Proposition~\ref{prop: not_abnormal_11} we may assume
$(h_1,h_2)\neq0$ a.e.  Lightlike extremals ($\lambda_0=0$) satisfy
$u(t)\in\partial C$ a.e., while $h(t)$ stays on the corresponding ray of
$\partial C^{*}$; as $h_3=-h_1$ is either identically zero or sign-constant,
no further switches occur.  

For timelike extremals ($\lambda_0=-1$) the maximised Hamiltonian
$\mathcal{H}(h,u)=\langle h,u\rangle+\nu(u)$ equals~$0$ and the maximum is
attained at some $u\neq0$ precisely when
$(-h_1,h_2)\in\partial\Omega^\diamond$, i.e.\ when
\eqref{eq:Lobachevski_plane_trigonometry_change} holds for some angle
$\eta$.  

The dynamics of $\eta$ follows from
$\dot\eta=-h_1\dot h_2+\dot h_1h_2=-h_3(h_1u_1+h_2u_2)=h_3$ and
\eqref{eq:vertical_subsystem}, giving
\[
   \ddot\eta
     =-a\,\cosh_\Omega\theta\,\sinh_{\Omega^\diamond}\eta
       +b\,\sinh_\Omega\theta\,\cosh_{\Omega^\diamond}\eta .
\]
Multiplying by $\dot\eta$ and integrating yields the first integral
$(\dot\eta)^2+a(\sinh_{\Omega^\diamond}\eta)^2
              -b(\cosh_{\Omega^\diamond}\eta)^2=E$,
which is exactly \eqref{eq:energy_inequality}.  
If the inequality is strict on $(\alpha,\beta)$, then $\dot\eta\neq0$ there
and separation of variables gives the stated quadrature.

If the extremal is singular along an edge $F$ of positive-measure time set,
then $(-h_1,h_2)$ must stay at the vertex of $\partial\Omega^\diamond$
dual to $F$, forcing $\dot\eta=0$ on that set and hence equality
in \eqref{eq:energy_inequality}.
\end{proof}

\end{document}